\documentclass[12pt,reqno]{article}
\usepackage{amssymb, amsthm, amsmath, amsfonts, stmaryrd, latexsym, txfonts, pxfonts, wasysym, graphicx, epsf}
\usepackage[v2]{xy}
\xyoption{curve}
\usepackage{fullpage}

\newcommand{\Rad}{\text{Rad}}
\newcommand{\Res}{\text{Quo}}
\newcommand{\im}{\text{im }}

\newcommand{\epz}{\varepsilon}
\newcommand{\eps}{\epsilon}
\newcommand{\phz}{\varphi}
\newcommand{\io}{\iota}
\newcommand{\A}{{\mathcal A}}
\newcommand{\Aa}{{\mathcal A}_a}
\newcommand{\Sa}{{\mathcal S}}
\newcommand{\Sal}{{\mathcal S}_{\text{a}}}
\newcommand{\Alg}{{\mathcal Alg}}
\newcommand{\Pp}{{\mathcal P}}
\newcommand{\iso}{\cong}
\newcommand{\id}{\text{id}}
\newcommand{\Mat}{\text{Mat}}
\newcommand{\dm}{/\!/}
\newcommand{\G}{\mathcal G}
\newcommand{\Hom}{\text{Hom}}
\newcommand{\Cx}{\mathbb C}

\newcommand{\tal}{\sim_{\text{alg}}}

\begin{document}

\begin{center}
\vskip 1cm{\LARGE\bf
Functors from Assocation Schemes}

\vskip 1cm
\large
Christopher French\\
Department of Mathematics and Statistics\\
Grinnell College\\
Grinnell, IA 50112\\
USA\\
{\tt frenchc@grinnell.edu}\\
\end{center}

\vskip .2in

\begin{abstract}
We construct a wide subcategory of the category of finite association schemes with a collection of
desirable properties.  Our subcategory has a first isomorphism theorem analogous to that
of groups.  Also, standard constructions taking schemes to groups (thin radicals and thin quotients)
or algebras (adjacency algebras) become functorial when restricted to our category.  We use our
category to give a more conceptual account for a result of Hanaki concerning products of characters
of association schemes; i.e. we show that the virtual representations of an association scheme
form a module over the representation ring of the thin quotient of the association scheme.

\end{abstract}

\newtheorem{thm}{Theorem}[section]
\newtheorem{cor}[thm]{Corollary}
\newtheorem{lem}[thm]{Lemma}
\newtheorem{question}[thm]{Question}
\newtheorem{con}[thm]{Conjecture}
\newtheorem{prop}[thm]{Proposition}
\theoremstyle{definition}
\newtheorem{defn}[thm]{Definition}
\newtheorem{notn}[thm]{Notation}
\newtheorem{rem}[thm]{Remark}
\newtheorem{example}[thm]{Example}
\newtheorem{cond}[thm]{Condition}

\section{Introduction}

Association schemes sit at the intersection of a variety of mathematical disciplines, having been studied
in statistics, yielding applications in coding theory, bearing relations to graph theory and
combinatorics, and, most recently, attracting the interest of group theorists.  Zieschang \cite{PHZ}
presents association schemes as a natural generalization of groups, explaining how one can define
subschemes, quotient schemes, morphisms between association schemes, kernels, and so on on.
In fact, by restricting one's attention to a certain class of association schemes, the ``thin" schemes,
one precisely recovers group theory.

With a definition of morphisms of association schemes, it is natural to investigate their categorical properties;
Hanaki \cite{Han} made such a study, referring to the category of association schemes by ${\mathcal AS}$.
Unfortunately this category is not as well-behaved as one might desire.  For example, morphisms
in general do not have images; thus there is no complete analogue of the first isomorphism theorem for groups.
Moreover, many of the standard operations that one applies to association schemes, such as taking
thin radicals or studying adjacency algebras, are not functors from ${\mathcal AS}$ to
another category.

In this paper, after giving some background on association schemes in Section \ref{sec:background},
we define, in Section \ref{sec:admissible}, a subcategory $\Sa$ of the category of
finite association schemes by adding an extra
requirement for morphisms; those morphisms satisfying our requirement will be called admissible
morphisms (see Definition \ref{def:admissible}).
Our subcategory is ``wide" in the sense that it contains all finite association schemes
as objects.  
Moreover, the category $\Sa$ has many of the desirable features absent from the traditional category.
For example, the image of an admissible morphism is a subscheme, and
in Section \ref{sec:first}, we can then prove a first isomorphism theorem for schemes.
We will also consider a quotient $\Sal$ of the category $\Sa$, in which one identifies
morphisms that coincide on scheme elements.

In Section \ref{sec:thin}, we investigate some constructions relating the category $\G$
of finite groups and our category $\Sa$ of association schemes.  
For example, given a finite group $G$, let $S(G)$ denote the corresponding thin
association scheme.  It is easy to check that
$S$ is a functor from $\G$ to $\Sa$; we let $S_a$ denote
the corresponding functor from $\G$ to $\Sal$.  In category theory,
whenever one is presented with a functor, it is natural to ask if it has a left or right adjoint.
In fact, as we show in Section \ref{sec:thin},
the functor $S_a$ has both a left and a right adjoint.  Given a finite scheme $S$, we let $\Rad(S)$
denote the group associated to the thin elements of $S$, and we show that
$\Rad$ is a functor from $\Sal$ to $\G$, which is right adjoint to $S_a$.
Likewise, given a scheme $S$, we can take the quotient of $S$ by its thin residue
to get a thin scheme, the thin quotient of $S$;
we let $\Res(S)$ denote the associated group.  Then $\Res$ is
a functor from $\Sal$ to $\G$, which
is left adjoint to $S_a$.  From a categorical point of view, such adjunctions help to explain the central
role of the thin radical and thin quotient constructions.  Moreover, the unit of the $(S_a,\Rad)$-adjunction and
the counit of the $(\Res,S_a)$-adjunctions are both isomorphisms, while the counit of the $(S_a,\Rad)$-adjunction
and the unit of the $(\Res,S_a)$-adjunctions are isomorphisms precisely when applied to thin schemes.
These facts give a categorical account of the group correspondence described by Zieschang in \cite[Section 5.5]{PHZ}.

In Section \ref{sec:adjalg}, we consider the adjacency algebra $\A(S)$ of an association scheme $S$.  We show that the
operation taking $S$ to $\A(S)$ is a functor from $\Sa$ to the category of algebras; again, this is not the case
if one works in the traditional category of association schemes.
In Section \ref{sec:product}, we investigate products of association schemes,
and we show that $\A$ takes products of schemes to tensor products
of algebras.  With this preparation, we show in Section \ref{sec:Hopf} that 
the adjacency algebra $\A(T)$ of a thin scheme $T$ is a Hopf algebra which is isomorphic to the group algebra
of $\Res(T)$; in general, if $S$ is a scheme, we show that
$\A(S)$ is a Hopf comodule over the group algebra of $\Res(S)$.  Using this, we show how one may view the complex
representations of the scheme $S$ as a module over the representation ring of the group $\Res(S)$.
This reformulates a result of Hanaki (\cite{Han2}) in a character-free setting.

\section{Background on Association Schemes}\label{sec:background}

We now consider some background about schemes; this material can be found in more detail in \cite{PHZ}.
For a set $X$, let $\Pp(X)$ denote the power set of $X$.

\begin{defn}\label{def:scheme} Given a finite set $X$, an association scheme, or scheme, on $X$, is a set
$S$ consisting of nonempty subsets $s\in \Pp(X\times X)$ and satisfying the following axioms.

\begin{enumerate}

	\item $S$ is a partition of $X\times X$;
	\item The subset $1_X:=\{(x,x):x\in X\}$ is an element of $S$;
	\item For each $s\in S$, the subset $s^*:=\{(x,y)\in X\times X:(y,x)\in s\}$ is in $S$.
	\item For each triple $p, q, r\in S$, there is a nonnegative integer $a_{pq}^r$ such that
	if $(x,z)\in r$, then \[\left|\{y\in X:(x,y)\in p\text{ and }(y,z)\in q\}\right|=a_{pq}^r.\]
	We call the integers $a_{pq}^r$ the {\it structure constants} of the scheme $S$.

\end{enumerate}
\end{defn}

\begin{example}\label{ex:group} Suppose $G$ is a finite group.  Let $S(G)$ denote the set of subsets of 
$G\times G$ of the form \[g\tilde{\:}:=\{(g_1,g_2)\in G\times G:g_2=g_1g\},\] for each $g\in G$.
Then it is easy to show that $S(G)$ is a scheme.  Moreover,
\begin{equation}\label{eqn:structure}
a_{g\tilde{\:} h\tilde{\:}}^{k\tilde{\:}}=\begin{cases} 1 & gh=k \\ 0 & \text{otherwise.}\end{cases}
\end{equation}
\end{example}

\begin{defn}\label{def:cprod}
Suppose $S$ is a scheme on a finite set $X$, and $P, Q\subseteq S$.  Then the {\it complex product} of $P$ and $Q$,
denoted $PQ$ is
$\{r\in S:a_{pq}^r>0\text{ for some }p\in P, q\in Q\}$.
\end{defn}

Complex product determines an associative
product on $\Pp(S)$.  If $p, q\in S$, we write $pq$, $pQ$ and $Pq$ for $\{p\}\{q\}$,
$\{p\}Q$ and $P\{q\}$ respectively.

\begin{defn}\label{def:closed}
If $S$ is a scheme on a finite set $X$, then 
a nonempty subset $T\subseteq S$ is a {\it closed subset} of $S$ if $TT=T$.
\end{defn}

If $T\subseteq S$ is closed, then the finiteness of $X$ implies that $t^*\in T$ whenever $t\in T$.

\begin{example}
If $G$ is a group, then a subset $T$ of the scheme $S(G)$ on $G$
is closed if and only if there is a subgroup $K\leq G$ such that $T=\{k\tilde{\:}:k\in K\}$.
\end{example}

\begin{defn}\label{def:morphism}
A morphism from a scheme $S$ on $X$ to a scheme $T$ on $Y$
is a function $\phi:X\cup S\to Y\cup T$ such that $\phi(X)\subseteq Y$,
$\phi(S)\subseteq T$, and whenever $s\in S$
and $(x_1,x_2)\in s$, we have $(\phi(x_1), \phi(x_2))\in \phi(s)$.
\end{defn}

A morphism is an isomorphism if it has an inverse,
or, equivalently, if it is a bijection.  We will typically denote the
restriction of $\phi$ to $X$ or $S$ by $\phi$, rather than $\phi|_X$ or $\phi|_S$.
When context permits, we will often just write $\phi:S\to T$, rather than $\phi:X\cup S\to Y\cup T$,
to indicate a morphism from a scheme $S$ to a scheme $T$.

\begin{rem}\label{rem:determine} The function $\phi:S\to T$
is determined by the function $\phi:X\to Y$.  Indeed, given $s\in S$, there is
a pair $(x_1, x_2)\in s$, and $\phi(s)$ must be the unique element of $T$ containing
$(\phi(x_1), \phi(x_2))$.
\end{rem}

\begin{defn}\label{def:ker} If $\phi$ is a morphism from a scheme $S$ to a scheme $T$, 
the kernel $\ker\phi$ of $\phi$ is defined to be $\{s\in S:\phi(s)=1\}$.
\end{defn}

By \cite[Lemma 5.1.2(ii)]{PHZ}, $\ker\phi$ is a closed subset of $S$.

\begin{defn}\label{def:normal}
A closed subset $T$ is {\it normal} if $pT=Tp$ for all $p\in S$, and {\it strongly normal} if $p^*Tp=T$
for all $p\in S$.  We call $pT$ and $Tp$ left and right cosets of $T$ respectively.
\end{defn}

Of course, if $T$ is strongly normal, then $Tp\subseteq pp^*Tp=pT$ and $pT\subseteq pTp^*p=Tp$,
so $T$ is normal.

\begin{notn}\label{notn:structure}
Given subsets $P, Q\subseteq S$ and $r\in R$, we write $a_{PQ}^r$ for $\sum_{p\in P, q\in Q} a_{pq}^r$.
As above, we write $a_{pQ}^r$ for $a_{\{p\}Q}^r$ or $a_{Pq}^r$ for $a_{P\{q\}}^r$.
\end{notn}

\begin{defn}\label{def:gcoset}  If $S$ is a scheme on a finite set $X$, $x\in X$, and $s\in S$, we let
$xs=\{y\in X:(x,y)\in s\}$.  If $T\subseteq S$ and $Y\subseteq X$, we let $YT=\cup_{y\in Y,t\in T} yt$.
We write $xT$ for $\{x\}T$ and $Yt$ for $Y\{t\}$.  If $T$ is a closed subset of $S$ and $x\in X$, then
$xT$ is called the {\it geometric coset} of $T$ containing $x$.
The set of geometric cosets of $T$ is denoted $X/T$.
\end{defn}

Of course, $xT$ contains $x$ since $1\in T$ and $x1=\{x\}$.

\begin{defn}\label{def:subscheme}
The {\it subscheme} of $T$ determined by $xT$, denoted $T_{xT}$, is the set of all nonempty subsets
of $xT\times xT$ of the form $t_{xT}:=t\cap (xT\times xT)$.
\end{defn}

Indeed, $T_{xT}$ is a scheme, cf. \cite[Theorem 2.1.8 (ii)]{PHZ}.

\begin{notn} Suppose $T$ is a closed subset of $S$.  For each $s\in S$,
let \[s^T=\{(xT,yT):(x',y')\in s\text{ for some }x'\in xT, y'\in yT\}.\]  For a subset $R\subseteq S$,
we let $R\dm T$ denote the set of all subsets $\{r^T\subseteq X/T\times X/T:r\in R\}$.
\end{notn}

\begin{defn}\label{def:quotient} The quotient scheme of $S$ by $T$ is the set $S\dm T$.
\end{defn}

Indeed, $S\dm T$ is a scheme on $X/T$, cf. \cite[Theorem 4.1.3]{PHZ}.

\begin{rem}\label{rem:quotient} If $p$ and $q$ are in $S$, then $p^T=q^T$ if and only if $TpT=TqT$.
See \cite[Lemma 4.1.1]{PHZ}.
\end{rem}

\begin{defn}\label{def:valency}
Given a scheme $S$ on a set $X$, the {\it valency} of an element $s\in S$, denoted $n_s$ is the number $a_{ss^*}^1$.

\end{defn}

Note that for any $x\in X$, $n_s$ counts the number of elements in $xs$.  Thus, for any $s$,
$n_s>0$ since $s$ is nonempty.
Since we are assuming throughout that $X$ is finite, it follows from \cite[Lemma 1.1.2(iii)]{PHZ}
that $n_s=n_{s^*}$ for any $s\in S$.

\begin{defn}\label{def:thin}
We say $s$ is {\it thin} if $ss^*=\{1\}$, and we say $S$ is a {\it thin scheme} if all of its elements are thin.
\end{defn}

By \cite[Lemma 1.5.1]{PHZ}, $s$ is thin if and only if $n_s=1$.  We let $O_{\vartheta}(S)$ denote the set
of thin elements of $S$; this is called the thin radical of $S$.  Since we assume $X$ to be finite,
$O_{\vartheta}(S)$ is closed by
\cite[Lemma 2.5.9]{PHZ}, and 
by \cite[Lemma 1.5.2]{PHZ}, if $p$ and $q$ are thin elements of a scheme $S$, then
$pq$ is a singleton set, so $pq$ consists of a single thin element.
Of course, $s\{1\}=s=\{1\}s$ for any $s\in S$.  For each thin element $s$,
$s^*$ is also thin, so $ss^*=\{1\}=s^*s$.  Therefore,
the complex product induces a group structure on the set of thin elements in a scheme $S$
on a finite set $X$.

\begin{defn}\label{def:radical} Suppose $S$ is a scheme on a finite set $X$.
Let $\Rad(S)$ denote the group of singleton subsets
consisting of a thin element of $S$, under complex product.
\end{defn}

\begin{example}\label{ex:not} Suppose $S$ is a thin scheme on a finite set $X$ with more than $2$ elements,
and let $T$ be the unique scheme
on $X$ with two elements $\{1_X,t\}$.  Let $\phi:X\cup S\to X\cup T$ be defined by
$\phi(x)=x$ for $x\in X$, $\phi(1_X)=1_X$, and $\phi(s)=t$ for $s\neq 1_X$.  Then it is easy to see that
$\phi$ is a morphism of schemes.  However, for $\{s\}\in \Rad(S)$, $s\neq 1_X$, we have
$\{\phi(s)\}=\{t\}$, and $\{t\}$ is not an element of $\Rad(T)$.  Thus, $\Rad$ fails to be a functor
from schemes to groups.
\end{example}

\begin{rem}\label{rem:intersection}
If $S$ is a scheme on a finite set $X$, it is easy to show that the intersection of a collection of strongly normal
closed subsets of $S$ is strongly normal.  The intersection of all strongly normal closed subsets of $S$
is called the thin residue of $S$, denoted $O^{\vartheta}(S)$.  By \cite[Lemma 4.2.5]{PHZ}, a closed subset $T$
is strongly normal if and only if the quotient scheme $S\dm T$ is thin.  In particular, $S\dm O^{\vartheta}(S)$
is thin.
\end{rem}

\begin{defn}\label{def:residue}  Suppose $S$ is a scheme on a finite set $X$.
Let $\Res(S)$ denote the group of singleton subsets of $S\dm O^{\vartheta}(S)$ under the
complex product.
\end{defn}

\section{Admissible morphisms}\label{sec:admissible}

In this section, we define our category $\Sa$ of schemes, as well as
a variant $\Sal$ of $\Sa$.  We show that classical
isomorphisms of schemes, inclusions of subschemes,
and quotients by normal closed subsets
are morphisms in $\Sa$.   We show that if $\phi$ is a morphism in $\Sa$ from a scheme
$S$ to a scheme $T$, then $\phi$ takes thin elements in $S$ to thin elements in $T$, and
we show that any morphism from a scheme to a thin scheme is a morphism in $\Sa$.  

\begin{defn}\label{def:admissible} Suppose $S$ is a scheme on a finite set $X$ and $T$ is a scheme on a
finite set $Y$.
A morphism of schemes $\phi$ from $S$ to $T$ is {\it admissible} if for any
$x\in X, y\in Y, s\in S$ such that
$(\phi(x),y)\in \phi(s)$, there exists $x'\in X$ such that $\phi(x')=y$ and $(x,x')\in s$.
\end{defn}

Zieschang \cite[p. 83]{PHZ} defines a {\it homomorphism} of schemes to be
a morphism $\phi:X\cup S\to Y\cup T$ such that
for any $x_1, x_2\in X$ and $s\in S$ such that $(\phi(x_1),\phi(x_2))\in \phi(s)$,
there exists $x_1', x_2'\in X$ such that $(x_1',x_2')\in s$, $\phi(x_1')=\phi(x_1)$, and
$\phi(x_2')=\phi(x_2)$.  

\begin{lem}\label{lem:admissiblehom} Any admissible morphism of schemes is a homomorphism.
\end{lem}

\begin{proof}
Suppose given $x_1, x_2\in X$ and $s\in S$ such that $(\phi(x_1), \phi(x_2))\in \phi(s)$.  Letting $y=\phi(x_2)$,
we get from Definition \ref{def:admissible} an element $x'\in X$ such that $\phi(x')=\phi(x_2)$ and $(x_1,x')\in s$.
Let $x_1'=x_1$ and $x_2'=x'$.
\end{proof}

As Zieschang observes \cite[Section 5.2]{PHZ}, the composition of two homomorphisms of schemes is not
in general a homomorphism.  For admissible morphisms, however, we have the following Lemma.

\begin{lem}\label{lem:category} The composition of any two admissible
morphisms is admissible.
\end{lem}

\begin{proof}
Suppose that $S$, $T$, and $U$ are schemes on $X$, $Y$, and $Z$ respectively,
and $\phi:S\to T$ and $\psi:T\to U$ are admissible morphisms of schemes.  Suppose $x\in X$, $z\in Z$, $s\in S$,
and $(\psi\phi(x),z)\in \psi\phi(s)$.  Then letting $y=\phi(x)$ and $t=\phi(s)$, we have
$(\psi(y),z)\in \psi(t)$, so by admissibility 
of $\psi$, there exists $y'\in Y$ such that $\psi(y')=z$ and $(y,y')\in t$.  That is, $(\phi(x),y')\in \phi(s)$.
By admissibility
of $s$, there exists $x'\in X$ such that $\phi(x')=y'$ and $(x,x')\in s$.  Thus, $\psi\phi(x')=\psi(y')=z$.
\end{proof}

\begin{lem}\label{lem:iso}
Any isomorphism of schemes is admissible.
\end{lem}

\begin{proof}
Suppose $S$ and $T$ are schemes on $X$ and $Y$ respectively, and $\phi:S\to T$ is an isomorphism of schemes.
Suppose, moreover, that $x\in X$, $y\in Y$, $s\in S$, and $(\phi(x),y)\in \phi(s)$.  Then $y=\phi(x')$ for a unique
$x'\in X$ since $\phi$ is an isomorphism.  Also, since $(\phi(x),\phi(x'))\in \phi(s)$ and $\phi$ is an isomorphism,
we have $(x,x')\in s$.
\end{proof}

For any scheme $S$ on $X$, the identity morphism is admissible by Lemma \ref{lem:iso},
and the composition of any two admissibile morphisms is admissible by Lemma \ref{lem:category}.
Thus, we have a category of association schemes.

\begin{defn}\label{def:S} Let $\Sa$ be the category whose objects are finite association schemes
and whose morphisms are admissible morphisms.
\end{defn}

We will also wish to consider a certain variant of the category $\Sa$.

\begin{defn}\label{def:alg-equiv} Suppose $S$ and $T$ are schemes on sets $X$ and $Y$.
We say two morphisms $\phi$ and $\phi'$ in $\Hom_{\Sa}(S,T)$ are {\it algebraically equivalent}
if $\phi(s)=\phi'(s)$ for all $s\in S$.  We then write $\phi\tal \phi'$.
\end{defn}

Note that if $\phi$ represents the equivalence class of an admissible morphism $\phz$,
then $\phi(s)$ is well-defined for $s\in S$.
Obviously, algebraic equivalence is an equivalence relation.  It is also easy to see that composition
respects algebraic equivalence in the following sense:  suppose $S$, $T$, and $U$ are schemes
on sets $X$, $Y$, and $Z$; suppose $\phi, \phi'\in \Hom_{\Sa}(S,T)$, while $\psi, \psi'\in \Hom_{\Sa}(T,U)$,
and suppose $\phi\tal \phi'$ and $\psi\tal \psi'$.  Then $\psi\circ\phi\tal\psi'\circ\phi'$.  Thus, we may
compose algebraic equivalence classes of morphisms, and we have the following category:

\begin{defn}\label{def:Salg}  Let $\Sal$ denote the category whose objects are finite association schemes
and whose morphisms are algebraic equivalence classes of admissible morphisms.  Let $F:\Sa\to \Sal$
denote the functor which is the identity on objects and takes a morphism to its equivalence class.
\end{defn}

We have seen that isomorphisms are always admissible.  We now show that inclusions of subschemes
and quotients by normal closed subsets are also admissible.

\begin{lem}\label{lem:inclusion} Suppose $T$ is a closed subset of a scheme $S$ on
a finite set $X$, and $z\in X$.
Let $\io_{zT}$ be the function from $zT\cup T_{zT}$ to $X\cup S$ defined by
\[\io_{zT}(x)=x\text{ for }x\in zT;\quad \io_{zT}(t_{zT})=t\text{ for }t_{zT}\in T_{zT}.\]
Then $\io_{zT}$ is an admissible morphism from $T_{zT}$ to $S$.
\end{lem}

\begin{proof}
It is easy to see that $\io_{zT}$ is a morphism of schemes.
If $(\io_{zT}(x),y)\in t$, then $x\in zT$ and $(x,y)\in t$, so $y\in xT=zT$.  Thus, $y=\io_{zT}(y)$ and $(x,y)\in t_{zT}$.
\end{proof}

\begin{lem}\label{lem:normal2} Suppose $S$ is a scheme on a finite set $X$ and $K$ is a normal closed subset of $S$.
Then the morphism $\pi_K:S\to S\dm K$ defined by $\pi_K(x)=xK$ and $\pi_K(s)=s^K$
is an admissible, surjective morphism with kernel $K$.
\end{lem}

\begin{proof} By \cite[Theorem 5.3.1]{PHZ}, $\pi_K$ is a surjective morphism with kernel $K$.  Suppose
$(\pi_K(x),yK)\in \pi_K(s)$;
that is, $(xK,yK)\in s^K$.  Then since $K$ is normal, $(x,x')\in s$ for some $x'\in yK$.  Thus,
$\pi_K(x')=x'K=yK$.
\end{proof}

\begin{cor}\label{cor:normal2} Suppose $S$ is a scheme on a finite set $X$, and $J$ and $K$ are normal closed subsets of $S$
with $J\subseteq K$.  Then there is an admissible, surjective morphism of schemes $\pi_K^J:S\dm J\to S\dm K$, defined by
$\pi_K^J(xJ)=xK$ and $\pi_K^J(s^J)=s^K$.  Moreover, $\pi_K^J\circ \pi_J=\pi_K$.
\end{cor}

\begin{proof}
Using \cite[Theorem 5.3.3]{PHZ}, we have an isomorphism $\psi:S\dm K\iso (S\dm J)\dm (K\dm J)$, and from
the proof of that theorem, we see that $\psi(xK)=(xJ)(K\dm J)$ and $\psi(s^K)=(s^J)^{K\dm J}$.  
By \cite[Lemma 4.2.4]{PHZ}, $K\dm J$ is normal in $S\dm J$, so by Lemma \ref{lem:normal2},
$\pi_{K\dm J}:S\dm J\to (S\dm J)\dm (K\dm J)$ is admissible.  By Lemmas \ref{lem:category}
and \ref{lem:iso}, $\psi^{-1}\circ \pi_{K\dm J}$ 
is admissible.  But \[\psi^{-1}(\pi_{K\dm J}(xJ))=\psi^{-1}((xJ)(K\dm J)=xK\quad\text{and}\quad
\psi^{-1}(\pi_{K\dm J}(s^J))=\psi^{-1}((s^J)^{K\dm J})=s^K.\]  Thus, $\pi_K^J=\psi^{-1}\circ\pi_{K\dm J}$,
which is admissible.  Since $\psi$ is an isomorphism and $\pi_{K\dm J}$ is surjective, $\pi_K^J$ is surjective.
Now, it is easy to check that $\psi\circ \pi_K=\pi_{K\dm J}\circ \pi_J$, so
\[\pi_K=\psi^{-1}\circ\pi_{K\dm J}\circ \pi_J=\pi_K^J\circ \pi_J.\]
\end{proof}

\begin{lem}\label{lem:relval} Suppose $S$ and $T$ are schemes on finite sets $X$ and $Y$, and 
$\phi\in \Hom_\Sa(S,T)$.  Then for each $s\in S$, there is a positive integer $n_s^\phi$ such that
if $(\phi(x_0),y)\in \phi(s)$, then there are $n_s^\phi$ elements $x\in X$ such that
$\phi(x)=y$ and $(x_0,x)\in s$.  In other words, $|\phi^{-1}(y)\cap x_0s|=n_s^\phi$.
\end{lem}

\begin{proof} By admissibility of $\phi$, there is at least one $x\in \phi^{-1}(y)\cap x_0s$.
Now, if $x'\in \phi^{-1}(y)$, and $(x',x)\in k$, then since $(y,y)=(\phi(x'),\phi(x))\in \phi(k)$,
we have $\phi(k)=1$, so $k\in \ker\phi$.  On the other hand, if $(x',x)\in k$, where $k\in \ker\phi$,
then $(\phi(x'),\phi(x))\in \phi(k)=1$, so $\phi(x')=\phi(x)=y$.  Thus,
$\phi^{-1}(y)\cap x_0s=\cup_{k\in K} (xk^*\cap x_0s)$, and this union is disjoint.
Since $(x_0,x)\in s$, $|xk^*\cap x_0s|=a_{sk}^s$, so $|\phi^{-1}(y)\cap x_0s|=\sum_{k\in K} a_{sk}^s
=a_{sK}^s$.  We let $n_s^\phi=a_{sK}^s$.
\end{proof}

\begin{cor}\label{cor:calculation} Suppose $S$ and $T$ are schemes on finite sets $X$ and $Y$, and 
$\phi\in \Hom_\Sa(S,T)$.  Then for each $s\in S$, we have $n_s^\phi=\frac{n_s}{n_{\phi(s)}}$.
\end{cor}

\begin{proof} Choose $x_0\in X$.  Then there are $n_{\phi(s)}$ elements $y\in Y$
such that $(\phi(x_0),y)\in \phi(s)$.  By Lemma \ref{lem:relval}, for each such element,
there are $n_s^\phi$ elements $x\in X$ such that $\phi(x)=y$ and $(x_0,x)\in s$.
Thus, there are $n_s^\phi\cdot n_{\phi(s)}$ pairs $(x,y)\in X\times Y$ such that
$(x_0,x)\in s$, $\phi(x)=y$, and $(\phi(x_0),y)\in \phi(s)$.

On the other hand, there are $n_s$ elements $x\in X$ such that $(x_0,x)\in s$.
For each such element, there is exactly one $y\in Y$ such that $\phi(x)=y$.
Thus, there are $n_s$ pairs $(x,y)\in X\times Y$ such that $(x_0,x)\in s$
and $\phi(x)=y$.  But these two conditions imply $(\phi(x_0),y)\in \phi(s)$.
Thus, $n_s=n_{\phi(s)}\cdot n_s^\phi$, proving the claim.
\end{proof}

\begin{cor}\label{cor:thin} Suppose $S$ and $T$ are schemes on finite sets $X$ and $Y$, and $\phi\in \Hom_\Sa(S,T)$.
If $s\in S$ is thin, $\phi(s)\in T$ is thin.
\end{cor}

\begin{proof}
By Corollary \ref{cor:calculation}, if $s\in S$ is thin, then
$n_{\phi(s)}\cdot n_s^\phi=n_s=1$, so $n_{\phi(s)}=1$, whence $\phi(s)$ is thin.
\end{proof}

\begin{lem}\label{lem:to-thin} Suppose $S$ and $T$ are schemes on finite sets $X$ and $Y$, and suppose
$T$ is thin.  Then any morphism $\phi$ from $S$ to $T$ is admissible.
In this case, $\ker(\phi)$ contains $O^{\vartheta}(S)$.
\end{lem}

\begin{proof}
Suppose $(\phi(x), y)\in \phi(s)$ for some $x\in X$, $y\in Y$, and $s\in S$.  Choose any $x'\in xs$.
Then $(\phi(x), \phi(x'))\in \phi(s)$ since $\phi$ is a morphism.  But $\phi(s)$
is thin, so this implies $y=\phi(x')$, which proves the first claim.

If $s\in S$ and $t\in s^*s$ then $\phi(t)\in \phi(s^*)\phi(s)=\phi(s)^*\phi(s)=1$ since $\phi(s)$ is thin.
Thus, $s^*s\subseteq \ker(\phi)$ for any $s\in S$.  But $O^{\vartheta}(S)$ is generated by the union
of the sets $s^*s, s\in S$; cf. \cite[Theorem 3.2.1.(ii)]{PHZ}.  In particular, 
since $\ker(\phi)$ is closed, $O^{\vartheta}(S)\subseteq \ker(\phi)$.

\end{proof}

\begin{cor}\label{cor:Sfunctor}
Given a finite group $G$, let $S(G)$ denote the thin scheme on the set $G$ defined in Example \ref{ex:group}.
Given a group homomorphism $\phi:G\to H$, let $S(\phi):G\cup S(G)\to H\cup S(H)$ denote the function taking
$g\in G$ to $\phi(g)\in H$ and taking $g\tilde{\:}$ to $\phi(g)\tilde{\:}$.  Then $S(\phi)$ is an admissible
morphism from the scheme $S(G)$ to the scheme $S(H)$.  Moreover, $S$ is a functor from the category
of finite groups to ${\Sa}$.
\end{cor}

\begin{proof} If $(g_1, g_2)\in g\tilde{\:}$, then $g_2=g_1g$, so $\phi(g_2)=\phi(g_1)\phi(g)$, whence
$(\phi(g_1),\phi(g_2))\in \phi(g)\tilde{\:}$.  That is, $(S(\phi)(g_1), S(\phi)(g_2))\in S(\phi)(g\tilde{\:})$,
meaning $S(\phi)$ is a morphism of schemes.  Since $S(H)$ is thin, $S(\phi)$ is admissible
by Lemma \ref{lem:to-thin}.  Now, it is immediate from the definitions that for a finite group $G$,
$S(\id_G)=\id_{S(G)}$ and if $\phi:G\to H$ and $\psi:H\to K$ are group homomorphisms, then
$S(\psi\circ\phi)=S(\psi)\circ S(\phi)$.
\end{proof}

\begin{defn}\label{def:SA}
Let $S_a$ denote the functor $F\circ S$ from the category of groups to ${\Sal}$.
(See Definition \ref{def:Salg}).
\end{defn}

\section{First isomorphism theorem}\label{sec:first}

In this section, we prove that an analogue of the first isomorphism theorem for groups holds in $\Sa$;
see Theorem \ref{thm:first-isomorphism}.  We recall that Zieschang proved
(\cite[Lemma 5.1.5]{PHZ}) that if $\phi$ is a morphism from a scheme $S$ on $X$, with kernel $K$,
then $\phi$ factors through a morphism from $S\dm K$ on $X/K$.  He also proved
(\cite[Theorem 5.3.2]{PHZ}) that if $\phi$ is a homomorphism from a scheme $S$ on $X$,
with kernel $K$, then $\phi$ factors through an injective homomorphism from the scheme
$S\dm K$ on $X/K$.  Here, we show that an admissible morphism $\phi$
from $S$ to $T$, with kernel $K$, factors through an isomorphism from the scheme
$S\dm K$ to a subscheme of $T$, the image of $\phi$.

First, we need to define the image of an admissible morphism.
Indeed, if $\phi$ is just a morphism from a scheme $S$ on $X$ to a scheme $T$ on $Y$, 
then $\phi(X)$ is not in general a coset of a closed subset of $T$, so the image of a
morphism is not itself a scheme.  However, we show in Lemma \ref{lem:image} that if $\phi$ is an admissible morphism,
then $\phi(X)$ is a coset of $\phi(S)$.  We then show in Lemma \ref{lem:normal}
that the kernel of an admissible morphism from $S$ to $T$ is a normal closed subset of $S$, just
as the kernel of a homomorphism of groups is a normal subgroup.  Finally, we show that
an admissible morphism $\phi$ from $S$ to $T$ induces an isomorphism
$\bar\phi:S\dm K\to \im\phi$ such that $\phi=\io_{\phi(X)}\circ \bar\phi\circ \pi_{K}$; that is,
the following diagram commutes:

\[\xymatrix
{
{S}
\ar[rr]^-{\phi}
\ar[d]_-{\pi_K} &&
{T}\\
{S\dm K} \ar[rr]^-{\bar\phi} &&
{\im\phi}\ar[u]_{\io_{\phi(X)}}
}
\]

If $S$ and $T$ are schemes on $X$ and $Y$, and $\phi\in \Hom_{\Sa}(S,T)$,
then for $R\subseteq S$, we let $\phi(R)=\{\phi(r):r\in R\}$, and for $W\subseteq X$, we let
$\phi(W)=\{\phi(w):w\in W\}$.

\begin{lem}\label{lem:image} Suppose $S$ and $T$ are schemes on finite sets $X$ and $Y$ respectively, and
$\phi\in \Hom_{\Sa}(S,T)$.  If $R\subseteq S$ is a closed subset, then $\phi(R)$ is a closed subset of $T$, and
for any geometric coset $W$ of $R$, $\phi(W)$ is a geometric coset of $\phi(R)$.
In particular, $\phi(S)$ is a closed subset in $T$, and $\phi(X)$ is a geometric coset of $\phi(S)$.
\end{lem}

\begin{proof} Suppose $R$ is closed in $S$, and
suppose $t, u\in \phi(R)$, so $t=\phi(p)$ and $u=\phi(q)$ for some $p, q\in R$.
Choose some $x\in X$, and let $y=\phi(x)$.  Suppose $v\in tu$.
Then for some $y', y''\in Y$, we have $(y,y')\in t$, $(y',y'')\in u$, and
$(y,y'')\in v$.  Since $(\phi(x),y')\in \phi(p)$, admissibility of $\phi$
implies that there is an $x'\in X$ such that $\phi(x')=y'$ and $(x,x')\in p$.  Similarly, since $(\phi(x'),y'')\in\phi(q)$, 
there is an $x''\in X$ such that $\phi(x'')=y''$ and $(x',x'')\in q$.  Now, $(x,x'')\in r$ for some $r\in pq$, and
since $R$ is closed, $r\in R$.
Since $\phi$ is a morphism, $(y,y'')=(\phi(x), \phi(x''))\in \phi(r)$, so $\phi(r)=v$.  Thus, $v\in \phi(R)$.  Since
$v$ was an arbitrary element in $tu$, we have $tu\subseteq \phi(R)$.  Since $t$ and $u$ were arbitrary
elements in $\phi(R)$, $\phi(R)$ is closed.

Now suppose $x\in W$, and again let $y=\phi(x)$.  We claim
that $\phi(W)=y(\phi(R))$.  Given $y'\in \phi(W)$, we have $y'=\phi(x')$ for some $x'\in W$.
Since $x, x'\in W$ and $W$ is a geometric coset of $R$,
we have $(x,x')\in r$ for some $r\in R$, so $(y,y')=(\phi(x),\phi(x'))\in \phi(r)$.  That is, $y'\in y(\phi(R))$, so
$\phi(W)\subseteq y(\phi(R))$.  Given $y'\in y(\phi(R))$, we have $(\phi(x),y')=(y,y')\in \phi(r)$ for some $r\in R$.
By admissibility of $\phi$, there is an $x'\in X$ such that $\phi(x')=y'$ and $(x,x')\in r$.  Since
$x'\in xr$, $x\in W$, and $W$ is a coset of $R$, we have $x'\in W$, so $y'\in \phi(W)$.  Thus,
$y(\phi(R))\subseteq \phi(W)$.
\end{proof}

\begin{notn} Suppose $S$ and $T$ are schemes on $X$ and $Y$ respectively, and $\phi\in \Hom_\Sa(S,T)$.
Let $U=\phi(S)$ and $Z=\phi(X)$.
By Lemma \ref{lem:image}, $U$ is closed in $T$, and $Z$ is a coset of $U$,
so for $x\in X$, $U_{Z}$ is a scheme on $Z$.  We will denote this scheme $\im\phi$.  Thus,
as a set, $\im\phi=\{\phi(s)_Z:s\in S\}$.
\end{notn}

\begin{lem}\label{lem:normal} Suppose $S$ and $T$ are schemes on finite sets
$X$ and $Y$ respectively, and $\phi\in
\Hom_\Sa(S,T)$.  Then $\ker\phi$ is a normal closed subset of $S$.
\end{lem}

\begin{proof} Let $K=\ker\phi$.  By \cite[Lemma 5.1.2(ii)]{PHZ}, $K$ is closed.
Now, suppose $s\in S$, and suppose $p\in sK$.  Then $p\in sk$ for some $k\in K$,
and we can find $x_1, x_2, x_3$ such that $(x_1,x_3)\in p$, $(x_1, x_2)\in s$ and $(x_2, x_3)\in k$.
Let $y_i=\phi(x_i)$ for $1\leq i\leq 3$.  Then $(y_2,y_3)\in \phi(k)=1$, so $y_2=y_3$, and $(y_1, y_2)\in \phi(s)$,
so $(y_1, y_3)\in \phi(s)$.  This implies $(\phi(x_3), y_1)\in \phi(s)^*=\phi(s^*)$.  Then since $\phi$ is admissible,
there is an $x_4\in X$ so that $(x_3,x_4)\in s^*$ and $\phi(x_4)=y_1=\phi(x_1)$.  Then $(x_4, x_3)\in s$ and
$(x_1, x_4)\in k'$ for some $k'\in \ker\phi$.  Therefore, since $(x_1, x_3)\in p$, we have $p\in k's\subseteq Ks$.
Thus, $sK\subseteq Ks$ for any $s\in S$.  By a similar argument, $Ks\subseteq sK$, so $K$ is normal.
\end{proof}

\begin{lem}\label{lem:induced} Suppose $S$ and $T$ are schemes on finite sets $X$ and $Y$ respectively,
$\phi\in \Hom_\Sa(S,T)$, and $K=\ker(\phi)$.  Define $\bar\phi:S\dm K\to \im\phi$ by $\bar\phi(xK)=\phi(x)$
and $\bar\phi(s^K)=\phi(s)_{\phi(X)}$.  Then $\bar\phi$ is an isomorphism of schemes.
\end{lem}

\begin{proof}
By Lemma \ref{lem:admissiblehom}, $\phi$ is a homomorphism, so by \cite[Theorem 5.3.2]{PHZ},
the function taking $xK$ to $\phi(x)\in Y$ and $s^K$ to $\phi(s)\in T$ is an injective homomorphism of schemes.
Thus, $\bar\phi$ is an injective morphism.  Since $\bar\phi$ is clearly surjective, $\bar\phi$ is an isomorphism.
\end{proof}

Note that for $x\in X$,
\[(\io_{\phi(X)}\circ \bar\phi\circ \pi_{K})(x)=(\io_{\phi(X)}\circ\bar\phi)(xK)=\io_{\phi(X)}(\phi(x))=\phi(x),\]
\[(\io_{\phi(X)}\circ \bar\phi\circ \pi_{K})(s)=(\io_{\phi(X)}\circ\bar\phi)(s^K)=
\io_{\phi(X)}\left(\phi(s)_{\phi(X)}\right)=\phi(s).\]
Thus, $\io_{\phi(X)}\circ\bar\phi\circ\pi_K=\phi$.

We summarize Lemmas \ref{lem:image}, \ref{lem:normal}, and \ref{lem:induced} in the following Theorem.

\begin{thm}\label{thm:first-isomorphism} (First Isomorphism Theorem for schemes).
Suppose $S$ and $T$ are schemes on finite sets
$X$ and $Y$ respectively, and $\phi\in \Hom_\Sa(S,T)$.  Let $K=\ker\phi$.
Then
\begin{enumerate}
	\item The kernel $K$ is a normal closed subset of $S$.
	\item The image $\im\phi$ is a subscheme of $T$.
	\item The admissible morphism $\phi$ induces an isomorphism $\bar\phi:S\dm K\to \im\phi$
	such that $\phi=\io_{\phi(X)}\circ \bar\phi\circ \pi_{K}$.  We have $\bar\phi(xK)=\phi(x)$ and
	$\bar\phi(s^K)=\phi(s)_{\phi(X)}$.
\end{enumerate}
\end{thm}

\begin{cor}\label{cor:first} Suppose $S$ and $T$ are schemes on $X$ and $Y$, and $S'$ and $T'$
are normal closed subsets of $S$ and $T$ respectively.  Now, suppose $\phi$ is an admissible
morphism from $S$ to $T$, and $\phi(S')\subseteq T'$.  Then there is a unique
morphism $\bar\phi:S\dm S'\to T\dm T'$ such that $\bar\phi\circ \pi_{S'}=\pi_{T'}\circ \phi$,
and $\bar\phi$ is admissible.
\end{cor}

\begin{proof}
First, $\pi_{T'}\circ \phi$ is admissible by Lemmas \ref{lem:category} and \ref{lem:normal2}.  Let
$K=\ker(\pi_{T'}\circ \phi)$.
By Theorem \ref{thm:first-isomorphism}, together with Lemmas \ref{lem:category} and \ref{lem:inclusion},
we have an admissible morphism $\psi:S\dm K\to T\dm T'$ such that $\psi\circ\pi_{K}=\pi_{T'}\circ \phi$,
defined by
\[\psi(xK)=\pi_{T'}(\phi(x))=\phi(x)T'\quad\text{and}\quad \psi(s^K)=(\pi_{T'}(\phi(s)))=\phi(s)^{T'}.\]
Now, if $s\in S'$, then $\phi(s)\in T'$, so $\pi_{T'}(\phi(s))=1$; that is, $S'\subseteq K$.  By Corollary
\ref{cor:normal2}, $\pi_{K}^{S'}$ is admissible and surjective, and $\pi_K=\pi_{K}^{S'}\circ \pi_{S'}$.  Let
$\bar\phi=\psi\circ\pi_{K}^{S'}$, which is admissible by Lemma \ref{lem:category}.  Thus, we have
\[\pi_{T'}\circ \phi=\psi\circ \pi_K=\psi\circ \pi_{K}^{S'}\circ \pi_{S'}=\bar\phi\circ \pi_{S'}.\]
Uniqueness of $\bar\phi$ follows from surjectivity of $\pi_{S'}$.
\end{proof}

\section{Thin radicals and thin quotients}\label{sec:thin}

In this section, let $\G$ denote the category of finite groups.
Recall from Definitions \ref{def:radical} and \ref{def:residue} that, to a scheme $S$, we can associate
two groups $\Rad(S)$ and $\Res(S)$, and recall from Example \ref{ex:group} that, to a group $G$, we
can associate a scheme $S(G)$.  In Corollary \ref{cor:Sfunctor}, we showed that $S$ extends to a functor
from $\G$ to $\Sa$.  In Definition \ref{def:SA}, we defined $S_a:\G\to \Sal$ as $F\circ S$.
We now show how to extend $\Rad$ and $\Res$ to functors from $\Sal$ to $\G$.
We will then show that the functor $S_a$ is right adjoint to $\Res$ and left adjoint to $\Rad$.

\begin{lem}\label{lem:Radfun} Suppose $S$ and $T$ are schemes on sets $X$ and $Y$, and
$\phi\in \Hom_{\Sa}(S,T)$.  Define $\Rad(\phi):\Rad(S)\to \Rad(T)$ by
\[\Rad(\phi)(\{s\})=\{\phi(s)\}.\]  Then $\Rad(\phi)$ is a group homomorphism.  If
$\psi\tal \phi$ then $\Rad(\phi)=\Rad(\psi)$.  With these definitions, $\Rad$ is a functor from
$\Sal$ to $\G$.
\end{lem}

\begin{proof}
By Corollary \ref{cor:thin}, if $s\in S$ is thin, then $\phi(s)$ is thin, so if 
$\{s\}\in \Rad(S)$, then $\{\phi(s)\}\in \Rad(T)$.
Now, suppose $p$ and $q$ are thin elements in $S$, and $r\in pq$.  Then $r$ is thin and $\{r\}=\{p\}\{q\}$.
Moreover, as above, $\phi(p)$ and $\phi(q)$ are thin, and $\phi(r)\in \phi(p)\phi(q)$, so  
$\{\phi(r)\}=\{\phi(p)\}\{\phi(q)\}$.  Thus,
\begin{equation*}
\begin{split}
\Rad(\phi)(\{p\}\{q\})&=\Rad(\phi)(\{r\})\\
&=\{\phi(r)\}\\
&=\{\phi(p)\}\{\phi(q)\}\\
&=\Rad(\phi)(\{p\})\Rad(\phi)(\{q\}),
\end{split}
\end{equation*}
so $\Rad(\phi)$ is a group homomorphism.

It follows immediately from Definition \ref{def:alg-equiv} that if $\phi\tal\psi$, then 
$\Rad(\phi)=\Rad(\psi)$.

If $S$ is a scheme on a finite set $X$, then for $\{s\}\in \Rad(S)$, we clearly 
have \[\Rad(\id_S)(\{s\})=\{\id_S(s)\}=\{s\}.\]  Likewise, if $S$, $T$, and $U$ are schemes
on $X$, $Y$, and $Z$, and if $\phi\in \Hom_{\Sa}(S,T)$ and $\psi\in \Hom_{\Sa}(T,U)$,
then we have
\begin{equation*}
\begin{split}
\Rad(\psi\circ\phi)(\{s\})&=\{\psi(\phi(s))\}\\
&=\Rad(\psi)(\{\phi(s)\})\\
&=\Rad(\psi)(\Rad(\phi)(\{s\})\\
&=(\Rad(\psi)\circ \Rad(\phi))(\{s\}).
\end{split}
\end{equation*}
\end{proof}

\begin{lem}\label{lem:Resfun}
Suppose $S$ and $T$ are schemes on sets $X$ and $Y$, and $\phi\in \Hom_\Sa(S,T)$.  Let $S'$ denote
$O^{\vartheta}(S)$ and let $T'$ denote $O^{\vartheta}(T)$.
Define $\Res(\phi):\Res(S)\to \Res(T)$ by \[\Res(\phi)(\{s^{S'}\})=\{\phi(s)^{T'}\}.\]
Then $\Res(\phi)$ is a group homomorphism.  
If $\psi\tal \phi$ then $\Res(\phi)=\Res(\psi)$.  With these definitions, $\Res$ is a functor from
$\Sal$ to $\G$.
\end{lem}

\begin{proof} We first need to prove that $\Res(\phi)$ is well-defined.
Since strongly normal subsets are normal, $T'$ is normal in $T$, so by Lemma \ref{lem:normal2},
the quotient morphism $\pi_{T'}:T\to T\dm T'$ is admissible.  Since $\Sa$ is a category,
the composition $\pi_{T'}\circ \phi$ is admissible.  Since $T\dm T'$ is thin by Remark
\ref{rem:intersection}, it follows from Lemma \ref{lem:to-thin} that $\ker(\pi_{T'}\circ \phi)$
contains $S'$.
Now, suppose $p, q\in S$, and $p^{S'}=q^{S'}$; we sill show that $\phi(p)^{T'}=\phi(q)^{T'}$.
Since $S'$ is normal, we have $p\in qs'$ for some
$s'\in S'$; cf. remark \ref{rem:quotient}.  Since $s'\in \ker(\pi_{T'}\circ\phi)$,
$\phi(s')\in \ker(\pi_{T'})=T'$ by Lemma \ref{lem:normal2}.
Thus, $\phi(p)\in \phi(q)\phi(s')\subseteq \phi(q)T'$, whence $\phi(p)T'\subseteq \phi(q)T'$.
By a symmetric argument, $\phi(q)T'\subseteq \phi(p)T'$, so $\phi(p)^{T'}=\phi(q)^{T'}$ by
Remark \ref{rem:quotient}.  Thus, $\Res(\phi)$ is well-defined.

Now, suppose $p, q\in S$, and choose $r\in pq$.  Then $r^{S'}\in p^{S'}q^{S'}$, so
$\{r^{S'}\}=\{p^{S'}\}\{q^{S'}\}$.  Also, $\phi(r)\in \phi(p)\phi(q)$, so $\phi(r)^{T'}\in
\phi(p)^{T'}\phi(q)^{T'}$, so $\{\phi(r)^{T'}\}=\{\phi(p)^{T'}\}\{\phi(q)^{T'}\}$.  Thus, we have
\begin{equation*}
\begin{split}
\Res(\phi)(\{p^{S'}\}\{q^{S'}\})&=\Res(\phi)(\{r^{S'}\})\\
&=\{\phi(r)^{T'}\}\\
&=\{\phi(p)^{T'}\}\{\phi(q)^{T'}\}\\
&=\Res(\phi)(\{p^{S'}\})\Res(\phi)(\{q^{S'}\}).
\end{split}
\end{equation*}
That is, $\Res(\phi)$ is a group homomorphism.

It follows immediately from Definition \ref{def:alg-equiv} that if $\phi\tal\psi$, then
$\Res(\phi)=\Res(\psi)$.

Now, for $\{s^{S'}\}\in \Res(S)$, we clearly have
\[\Res(\id_S)(\{s^{S'}\})=\{\id_S(s)^{S'}\}=\{s^{S'}\}.\]
Likewise, if $S$, $T$, and $U$ are schemes
on finite sets $X$, $Y$, and $Z$, and if $\phi\in \Hom_{\Sa}(S,T)$ and $\psi\in \Hom_{\Sa}(T,U)$,
then we have
\begin{equation*}
\begin{split}
\Res(\psi\circ\phi)(\{s^{S'}\})&=\psi(\phi(s))^{U'}\\
&=\Res(\psi)(\phi(s)^{T'})\\
&=\Res(\psi)(\Res(\phi)(s^{S'}))\\
&=(\Res(\psi)\circ \Res(\phi))(s^{S'}).
\end{split}
\end{equation*}
\end{proof}

Recall that an adjunction between categories ${\mathcal C}$ and ${\mathcal D}$ consists of
a pair of functors $L:{\mathcal C}\to {\mathcal D}$ and $R:{\mathcal D}\to {\mathcal C}$,
together with a natural isomorphism of set-valued functors
\[\eta:\Hom_{\mathcal D}(L(-),-)\to \Hom_{\mathcal C}(-,R(-)).\]
We then say that $L$ is the left adjoint functor and $R$ is the right adjoint functor.
We write $\eta_C^D$ for the function
\[\Hom_{\mathcal D}(LC,D)\to \Hom_{\mathcal C}(C,RD).\]
If $C$ is an object of ${\mathcal C}$, then $\eta_C^{LC}(\id_{LC})$
is a morphism in $\mathcal C$ from $C$ to $RLC$, called the unit of the $(L,R)$-adjunction.
Similarly, if $D$ is an object of ${\mathcal D}$, then
$(\eta_{RD}^D)^{-1}(\id_{RD})$ is a morphism in ${\mathcal D}$ from $LRD$ to $D$,
called the counit of the $(L,R)$-adjunction.  See MacLane \cite{MacLane} for more
details on adjunctions.
In the next two Propositions, we will define two adjunctions between $\G$ and $\Sal$.

\begin{prop}\label{prop:Radadjoint} There is an adjunction between the categories
$\G$ and $\Sal$ for which the
functor $S_a:\G\to \Sal$ is the left adjoint functor and the functor $\Rad:\Sal\to \G$ is the right adjoint
functor.
\end{prop}

\begin{proof} 
We will first define a function \[\eta_T^G:\Hom_{\Sal}(S_a(G),T)\to \Hom_{\G}(G,\Rad(T))\]
for each finite group $G$ and finite scheme $T$.
Then, we will show that these functions define a natural transformation $\eta$.
Then, we will define functions
\[\epz_T^G:\Hom_{\G}(G,\Rad(T))\to \Hom_{\Sal}(S_a(G),T),\]
and finally we will show that $\epz_T^G$
and $\eta_T^G$ are inverses, implying that $\eta_T^G$ is a bijection for each $G$ and $T$,
so that $\eta$ is a natural isomorphism.

Given $\phi\in \Hom_{\Sal}(S_a(G), T)$, Definition \ref{def:alg-equiv} implies that
$\phi(g\tilde{\:})$ is well-defined for $g\in G$, and Corollary \ref{cor:thin} implies that
$\phi(g\tilde{\:})$ is thin.  Let $\eta_T^G(\phi)(g)=\{\phi(g\tilde{\:})\}\in \Rad(T)$.
From Equation (\ref{eqn:structure}) in Section \ref{sec:background},
we have $\{(g_1g_2)\tilde{\:}\}=g_1\tilde{\:}g_2\tilde{\:}$ for any $g_1, g_2\in G$,
so \[\{\phi((g_1g_2)\tilde{\:}\}=\phi(g_1\tilde{\:}g_2\tilde{\:})\subseteq
\phi(g_1\tilde{\:})\phi(g_2\tilde{\:}).\]
Since the product of two thin elements is a singleton set consisting of a thin element,
we actually have equality above, whence
\[\eta_T^G(\phi)(g_1g_2)=\{\phi((g_1g_2)\tilde{\:}\}=
\{\phi(g_1\tilde{\:})\}\{\phi(g_2\tilde{\:})\}=\eta_T^G(\phi)(g_1)\cdot\eta_T^G(\phi)(g_2),\] so
$\eta_T^G(\phi)\in \Hom_{\G}(G,\Rad(T)).$

Next, we show that $\eta_T^G$ is natural.  In other words, we need to show that if 
$\phi\in \Hom_{\Sal}(S_a(G_1),T_1)$, $\psi\in \Hom_{\Sal}(T_1,T_2)$
and $\chi\in \Hom_{\G}(G_2,G_1)$, then \[\Rad(\psi)\circ \eta_{T_1}^{G_1}(\phi)\circ \chi=
\eta_{T_2}^{G_2}(\psi\circ \phi\circ S_a(\chi)):G_2\to \Rad(T_2).\]
For $g\in G_2$, we have
\begin{equation*}
\begin{split}
(\Rad(\psi)\circ \eta_{T_1}^{G_1}(\phi)\circ \chi)(g)& 
=\Rad(\psi)\left(\{\phi(\chi(g)\tilde{\:})\}\right)\\
&=\{(\psi\circ\phi)(\chi(g)\tilde{\:})\}\\
&=\{(\psi\circ\phi\circ S_a(\chi))(g\tilde{\:})\}\\
&=\eta_{T_2}^{G_2}(\psi\circ\phi\circ S_a(\chi))(g).
\end{split}
\end{equation*}

Now we will define \[\epz_T^G:\Hom_{\G}(G,\Rad(T))\to \Hom_{\Sal}(S_a(G),T).\]
Given $\phi\in \Hom_{\G}(G,\Rad(T))$, we will define an admissible morphism
of schemes $\eps_T^G(\phi)$ from $S(G)$ to $T$, and take $\epz_T^G(\phi)$
to be the algebraic equivalence class of $\eps_T^G(\phi)$.  (Note the difference between
$\eps$ and $\epz$.)  Suppose $T$ is a scheme
on $Y$.
Choose a point $y_*\in Y$; then for each $g\in G$,
$\phi(g)=\{t\}$ for some thin element $t\in T$, so $y_*\phi(g)$ is a singleton set.
For $g\in G$, let $\eps_T^G(\phi)(g)$ be the unique element in $y_*\phi(g)$, and for $g\tilde{\:}\in S_a(G)$,
let $\eps_T^G(\phi)(g\tilde{\:})$ be the unique (thin) element in the singleton set $\phi(g)$.

We now show that $\eps_T^G(\phi)$ is an admissible morphism of schemes.
Suppose $(g_1, g_2)\in g\tilde{\:}$, and let $y_1$ and $y_2$ be the unique elements in $y_*\phi(g_1)$
and $y_*\phi(g_2)$, and let $t$ be the unique element in $\phi(g)$.
Then $g_2=g_1g$, so $\phi(g_2)=\phi(g_1)\phi(g)$.  Thus,
\[\{y_2\}=y_*\phi(g_2)=y_*\phi(g_1)\phi(g)=y_1\phi(g)=y_1t,\] so $(y_1, y_2)\in t$.
This shows that $\eps_T^G(\phi)$ is a morphism of schemes.
If $(\eps_T^G(\phi)(g_1),y)\in \eps_T^G(\phi)(g\tilde{\:})$,
then letting $y_*\phi(g_1)=\{y_1\}$, we have $y\in y_1\phi(g)\subseteq y_*\phi(g_1)\phi(g)=y_*\phi(g_1g)$.
Thus, $\{y\}=y_*\phi(g_1g)=\eps_T^G(\phi)(g_1g)$, and $(g_1,g_1g)\in g\tilde{\:}$.  Therefore,
$\eps_T^G(\phi)$ is admissible.

Now, if $\phi\in \Hom_{\Sa}(S_a(G),T)$, then
$\epz_T^G(\eta_T^G(\phi))(g\tilde{\:})$ is the unique thin element in the singleton set
$\eta_T^G(\phi)(g)=\{\phi(g\tilde{\:})\}$, so by Definition \ref{def:alg-equiv}
$(\epz_T^G\circ\eta_T^G)(\phi)=\phi$.  On the other hand, if $\phi\in \Hom_{\G}(G,\Rad(T))$,
then $\eta_T^G(\epz_T^G(\phi))(g)=\{\epz_T^G(\phi)(g\tilde{\:})\}=\phi(g)$, so
$(\eta_T^G\circ\epz_T^G)(\phi)=\phi$.  Thus,
$\eta_T^G$ and $\epz_T^G$ are inverse functions, so $\eta_T^G$ is a bijection for each $G,T$.

\end{proof}

\begin{rem}\label{rem:unitSR} 
For a group $G$, the unit of the $(S_a,\Rad)$-adjunction given in Proposition \ref{prop:Radadjoint} is the
morphism $\eta_{S_a(G)}^G(\id_{S_a(G)})\in \Hom_{\G}(G,\Rad(S_a(G)))$.  This morphism
takes $g\in G$ to $\{g\tilde{\:}\}$, and is clearly an isomorphism.

For a scheme $T$, the
counit of the adjunction is a morphism $\phi\in \Hom_{\Sal}(S_a(\Rad(T)),T)$
satisfying $\eta_T^{\Rad(T)}(\phi)=\id_{\Rad(T)}$.
Then, for an element
$\{t\}\in \Rad(T)$, we have \[\{t\}=\eta_T^{\Rad(T)}(\phi)(\{t\})=\{\phi(\{t\}\tilde{\:})\},\]
so $\phi(\{t\}\tilde{\:})=t$.  Thus, the kernel of $\phi$ is trivial.  We may choose a representative
$\phz$ of the algebraic equivalence class of $\phi$ (note the difference between
$\phi$ and $\phz$).  Then, the image of $\phz$ is a subscheme of $T$ defined by a coset
of the thin radical of $T$; which subscheme one
obtains depends on the choice of $\phz$ representing $\phi$. 
In particular, if $T$ is thin, then
any representative of the counit of the adjunction is an isomorphism of schemes,
by Theorem \ref{thm:first-isomorphism}.  Note, however, that there does not exist a canonical admissible
morphism from the scheme $S(\Rad(T))$ to $T$, only an algebraic equivalence class of such
morphisms.
\end{rem}

\begin{prop}\label{prop:Resadjoint}
There is an adjunction between the categories $\G$ and $\Sal$ for which the
functor $\Res:\Sal\to \G$ is the left adjoint
functor and the functor $S_a:\G\to \Sal$ is the right adjoint functor.
\end{prop}

\begin{proof} As in the proof of Proposition \ref{prop:Radadjoint},
we will define a function \[\eta_G^T:\Hom_{\G}(\Res(T),G)\to \Hom_{\Sal}(T,S_a(G))\]
for each finite group $G$ and finite scheme $T$.
Then, we will show that these functions define a natural transformation $\eta$.
Then, we will define functions
\[\epz_G^T:\Hom_{\Sal}(T,S_a(G))\to \Hom_{\G}(\Res(T),G),\]
and finally we will show that $\epz_G^T$
and $\eta_G^T$ are inverses, implying that $\eta_G^T$ is a bijection for each $G$ and $T$.

Given a scheme $T$ on a finite set $Y$, let $T'=O^{\vartheta(T)}$.
Given $\phi\in \Hom_{\G}(\Res(T),G)$, we will define $\eta_G^T(\phi)$ by defining
an admissible morphism $\alpha$ from $T$ to $S(G)$; we then take $\eta_G^T(\phi)$ to
be the algebraic equivalence class of $\alpha$.
To this end, choose $y_*\in Y$, and for each $y\in Y$, let $\tau(y)\in T$ denote the element
containing $(y_*,y)$.  Now, for $y\in Y$,
let $\alpha(y)=\phi\left(\{\tau(y)^{T'}\}\right)$, and for $t\in T$,
let $\alpha(t)=(\phi(\{t^{T'}\}))\tilde{\:}$.  We now show that $\alpha$ is a morphism
of schemes.
Suppose $(y_1,y_2)\in t$; then $\tau(y_2)\in \tau(y_1)t$, so
$\{\tau(y_2)^{T'}\}=\{\tau(y_1)^{T'}\}\{t^{T'}\}$.  Since $\phi$ is a
group homomorphism, we have
\[\phi(\{\tau(y_2)^{T'}\})=\phi\left(\{\tau(y_1)^{T'}\}\right)\phi\left(\{t^{T'}\}\right).\]
Thus,
\[\left(\phi\left(\{\tau(y_1)^{T'}\}\right),\phi(\{\tau(y_2)^{T'}\}\right)\in \phi\left(\{t^{T'}\}\right)\tilde{\:},\]
or $(\alpha(y_1),\alpha(y_2))\in \alpha(t)$.  Thus, $\alpha$ is a morphism of schemes, and
since $S_a(G)$ is thin, $\alpha$ is admissible by Lemma \ref{lem:to-thin}.  Now, from
our definition of $\alpha(t)$, we see that the algebraic equivalence class of $\alpha$ does not depend on
our choice of $y_*$.  Thus, we may let $\eta_G^T(\phi)$ denote the algebraic equivalence
class of $\alpha$.

Now, we show that $\eta_G^T$ is natural.  Assume $\phi\in \Hom_{\G}(\Res(T_1),G_1)$,
$\psi\in \Hom_{\Sal}(T_2,T_1)$, and $\chi\in \Hom_{\G}(G_1,G_2)$.  We must show
\[S_a(\chi)\circ \eta_{G_1}^{T_1}(\phi)\circ \psi=\eta_{G_2}^{T_2}(\chi\circ\phi\circ\Res(\psi)).\]
Since we are working in the category $\Sal$, it suffices to show that both sides of the above
equation coincide when applied to some $t\in T_2$.  We have
\begin{equation*}
\begin{split}
(S_a(\chi)\circ\eta_{G_1}^{T_1}(\phi)\circ\psi)(t)
&=S_a(\chi)(\eta_{G_1}^{T_1}(\phi)(\psi(t))\\
&=S_a(\chi)\left(\left(\phi\left(\{\psi(t)^{T_1'}\}\right)\right)\tilde{\:}\right)\\
&=\left(\chi(\phi(\{\psi(t)^{T_1'}\}))\right)\tilde{\:}\\
&=\left((\chi\circ\phi)\left(\{\psi(t)^{T_1'}\}\right)\right)\tilde{\:}\\
&=\left(\left(\chi\circ\phi\circ\Res(\psi)\right)\left(\{t^{T_2'}\}\right)\right)\tilde{\:}\\
&=\eta_{G_2}^{T_2}(\chi\circ\phi\circ\Res(\psi))(t).
\end{split}
\end{equation*}
Now, we define \[\epz_G^T:\Hom_{\Sal}(T,S_a(G))\to \Hom_{\G}(\Res(T),G).\]
For this, we adopt the following notation.  Given a group $G$, and given an element
$s\in S(G)$, we write $s\hat{\:}$ for the element $g\in G$ such that $g\tilde{\:}=s$.
Suppose given a morphism $\phi\in \Hom_{\Sal}(T,S_a(G))$,
and let $\varphi$ denote a representative of $\phi$ in $\Hom_{\Sa}(T,S(G))$.
As above, let $T'=O^{\vartheta}(T)$.  Note that since $S_a(G)$ is thin, 
$O^{\vartheta}(S_a(G))$ is trivial, and $T'\subseteq \ker(\varphi)$
by Lemma \ref{lem:to-thin}.

By Corollary \ref{cor:first}, there is an admissible morphism
\[\bar\varphi:T\dm T'=T\dm O^{\vartheta(T)}\to S_a(G)\dm O^{\vartheta}(S_a(G))\iso S_a(G)\]
such that $\bar\varphi\circ\pi_{T'}=\varphi$.  Thus, given $t\in T$, $\varphi(t)$
depends only on $t^{T'}$.  Given $\{t^{T'}\}\in \Res(T)$, let
$\epsilon_G^T(\phi)\left(\{t^{T'}\}\right)=\varphi(t)\hat{\:}\in G$.  Of course,
$\varphi(t)$ is independent of the choice of representative $\varphi$ of $\phi$,
so $\epsilon_G^T(\phi)$ is a well-defined function from $\Res(T)$ to $G$.

We now show that $\epsilon_G^T(\phi)$ is a group homomorphism.
Suppose given $\{t_1^{T'}\}, \{t_2^{T'}\}\in
\Res(T)$.  Choose $t_3\in t_1t_2$.  Then $\{t_3^{T'}\}=\{t_1^{T'}\}\{t_2^{T'}\}$.
Let $g_i=\epsilon_G^T(\phi)(t_i^{T'})$ for $i\in \{1,2,3\}$.  Then
$g_i\tilde{\:}=\varphi(t_i)$.  Since $t_3\in t_1t_2$, we have $\varphi(t_3)\in
\varphi(t_1)\varphi(t_2)$, so $g_3\tilde{\:}\in g_1\tilde{\:}g_2\tilde{\:}$,
or equivalently, $t_3=g_1g_2$.  Thus, we have
\[\epsilon_G^T(\phi)\left(\{t_1^{T'}\}\{t_2^{T'}\}\right)=
\epsilon_G^T(\phi)\left(\{t_3^{T'}\}\right)=
g_3=g_1g_2=\epsilon_G^T(\phi)\left(\{t_1^{T'}\}\right)
\epsilon_G^T(\phi)\left(\{t_2^{T'}\}\right).\]
Thus, $\epsilon_G^T(\phi)$ is a group homomorphism.

Finally, we must show that $(\eta_G^T\circ\epsilon_G^T)(\phi)=\phi$ for
$\phi\in \Hom_{\Sal}(T,S_a(G))$ and $(\epsilon_G^T\circ\eta_G^T)(\psi)=\psi$
for $\psi\in \Hom_{\G}(\Res(T),G)$.  For the first equation, given $t\in T$, we have
\[(\eta_G^T\circ\epsilon_G^T)(\phi)(t)=
\eta_G^T(\epsilon_G^T(\phi))(t)
=\left(\epsilon_G^T(\phi)\left(\{t^{T'}\}\right)\right)\tilde{\:}\\
=(\varphi(t)\hat{\:})\tilde{\:}=\phi(t).\]

Since we are working in the category $\Sal$, this suffices to show $(\eta_G^T\circ\epsilon_G^T)(\phi)=\phi$.
For the second equation, suppose given $\{t^{T'}\}\in \Res(T)$.  Let $\alpha$ denote
a representative in $\Hom_{\Sa}(T,S(G))$ of $\eta_G^T(\psi)$.  Then, we have
\[((\epsilon_G^T\circ\eta_G^T)(\psi))(\{t^{T'}\})
\alpha(t)\hat{\:}
=(\psi(\{t^{T'}\}))\tilde{\:}\hat{\:}
=\psi(\{t^{T'}\}).\]
\end{proof}

\begin{rem}\label{rem:unitSQ}
For a group $G$, the counit of the adjunction given in Proposition \ref{prop:Resadjoint} is the
morphism $\phi\in \Hom_{\G}(\Res(S_a(G)),G)$ satisfying $\eta_{G}^{S_a(G)}(\phi)=\id_{S_a(G)}$.
That is, \[g\tilde{\:}=\eta_G^{S_a(G)}(\phi)(g\tilde{\:})=(\phi(\{g\tilde{\:}\}))\tilde{\:},\]
so $\phi(\{g\tilde{\:}\})=g$.  Here, we write $\{g\tilde{\:}\}$ rather than
$\{(g\tilde{\:})^{O^{\vartheta}(S_a(G))}\}$ for an element in $\Res(S_a(G))$
since $O^{\vartheta}(S_a(G))$ is trivial.  Clearly, $\phi$ is an isomorphism of groups.

For a scheme $T$, with $T'=O^{\vartheta}(T)$,
the unit of the adjunction is a morphism $\psi:=\eta_{\Res(T)}^{T}(\id_{\Res(T)})$
in $\Hom_{\Sal}(T,S_a(\Res(T)))$, which takes $t$ to $\{t^{T'}\}\tilde{\:}\in S_a(\Res(T))$.
Thus, the kernel of a the unit of the adjunction is the thin residue of $T$.
Now, $\psi$ can be represented by a surjective morphism in
$\Hom_{\Sa}(T,S(\Res(T)))$.  In particular, if $T$ is thin, then
any representative of the unit of the adjunction is an isomorphism of schemes.
But again, we caution that there does not exist
a canonical admissible morphism from $T$ to $S(\Res(T))$, only an algebraic equivalence class
of such morphisms.
\end{rem}

\section{Adjacency algebras of association schemes}\label{sec:adjalg}

Suppose $S$ is a scheme on a finite set $X$.  Let $\Mat_X(\Cx)$ denote the set of matrices with entries
in $\Cx$ and with rows and columns indexed by $X$.
Then each $s\in S$ determines a matrix $\sigma_s\in \Mat_X(\Cx)$, called the {\it adjacency matrix} of $s$.
The $(x,y)$-entry
of $\sigma_s$ is $1$ if $(x,y)\in s$, and $0$ otherwise.  Since each $s$ is nonempty, $\sigma_s$
is nonzero.  Since $S$ is a partition of $X\times X$,
we have $\sum_{t\in T} \sigma_s=J$, where $J\in \Mat_X(\Cx)$ consists entirely of $1$'s.
Of course, $\sigma_{1_X}$ is the identity matrix, and $\sigma_{s^*}$ is the transpose of $\sigma_s$ for
any $s\in S$.  Finally, if $(x,y)\in r$, then for any $p, q\in S$
\[(\sigma_p\circ \sigma_q)(x,y)=\sum_z \sigma_p(x,z)\cdot \sigma_q(z,y)=
\left|\{z\in X:(x,z)\in p\text{ and } (z,y)\in q\}\right|=a_{pq}^r.\]
Therefore, \begin{equation}\label{eqn:sigma}
\sigma_p\circ\sigma_q=\sum_{r\in S} a_{pq}^r \sigma_r.
\end{equation}
It follows that the adjacency matrices span a subalgebra of $\Mat_X(\Cx)$.  This subalgebra
is semisimple since it is closed under transposition.

\begin{defn}\label{def:algebra} Suppose $S$ is a scheme on
a finite set $X$.  The {\it adjacency algebra} of $S$, denoted $\A(S)$, is the span of
the adjacency matrices $\sigma_s\in \Mat_X(\Cx)$.
\end{defn}

\begin{rem}\label{rem:adjacency}
For a finite group $G$, let $\mathbb C G$ denote the group algebra of $G$.
Then the linear transformation $\mathbb C G\to \A(S(G))$ which
takes $g\in G$ to $\sigma_{g\tilde{\:}}$ is an isomorphism of algebras.
Indeed, bijectivity is obvious, and if $gh=k$ in $G$, then Equation \eqref{eqn:structure} from Section
\ref{sec:background}, together with Equation \eqref{eqn:sigma} above,
implies that $\sigma_{g\tilde{\:}}\circ \sigma_{h\tilde{\:}}=\sigma_{k\tilde{\:}}.$
\end{rem}

Now let $\Alg$ denote the category of (unital) algebras over $\Cx$.
We next show how to extend $\A$ to a functor from $\Sal$ to $\Alg$ (see Corollary \ref{cor:Algfunctor} below).
Recall from Lemma \ref{lem:relval} and Corollary \ref{cor:calculation}
that if $\phi$ is an admissible morphism from a scheme $S$
on $X$ to a scheme $T$ on $Y$, and $(\phi(x_0),y)\in \phi(s)$, then there are $n_s^\phi=\frac{n_s}{n_{\phi(s)}}$
elements $x\in x_0s$ such that $\phi(x)=y$.

\begin{lem}\label{lem:algebra} Suppose $S$ and $T$ are schemes on $X$ and $Y$, and
$\phi\in \Hom_\Sa(S,T)$.  Then for any $p,q\in S$ and $t\in T$, we have
\[\sum_{s:\phi(s)=t} a_{pq}^{s}n_s^{\phi} =a_{\phi(p)\phi(q)}^{t}n_p^\phi n_q^\phi.\]
\end{lem}

\begin{proof}
From Lemma \ref{lem:image}, $\phi(S)$ is closed.  Therefore, if
$t\notin \phi(S)$, then the left side of our equation is clearly $0$, and since
$t\notin \phi(p)\phi(q)$ for any $p,q\in S$, we have $a_{\phi(p)\phi(q)}^t=0$, so the right side of the
equation is $0$ as well.  Thus, it suffices to show that for any $p, q, r\in S$, we have
\[\sum_{s:\phi(s)=\phi(r)} a_{pq}^{s}n_s^{\phi} =a_{\phi(p)\phi(q)}^{\phi(r)}n_p^\phi n_q^\phi.\]

Suppose $(x_1,x_3)\in r$.  Let $y_1=\phi(x_1)$, and let $y_3=\phi(x_3)$, so $(y_1,y_3)\in \phi(r)$.
Let $\Omega_0$ be the set of pairs $(x_2',x_3')\in q$ such that $(x_1,x_2')\in p$ and
$x_3'\in \phi^{-1}(y_3)$.  We compute $|\Omega_0|$.
If $(x_2',x_3')\in \Omega_0$, then $(x_1,x_3')\in s$ for some $s$
such that $\phi(s)=\phi(r)$, since $(y_1,y_3)\in \phi(r)$.  Moreover,
for each $s\in S$ such that $\phi(s)=\phi(r)$, we have
$(\phi(x_1),y_3)\in \phi(s)$, so there are $n_s^\phi$ elements $x_3'\in \phi^{-1}(y_3)\cap x_1s$.
For each of these choices of $x_3'$, there are $a_{pq}^s$ elements $x_2'$ such that
$(x_1,x_2')\in p$ and $(x_2',x_3')\in q$.  Thus, we have
\begin{equation}\label{eqn:omega0} |\Omega_0|=\sum_{s\in S:\phi(s)=\phi(r)} a_{pq}^{s}n_s^\phi.\end{equation}

Now let $\Omega_1$ denote the set of elements $y_2\in Y$ such that $(y_1,y_2)\in \phi(p)$ and
$(y_2,y_3)\in \phi(q)$.  Thus we have
\begin{equation}\label{eqn:omega1} |\Omega_1|=a_{\phi(p)\phi(q)}^{\phi(r)}.\end{equation}
For each $(x_2',x_3')\in \Omega_0$, we see that $\phi(x_2')\in \Omega_1$.
Thus, we have a function $f:\Omega_0\to \Omega_1$, where $f(x_2',x_3')=\phi(x_2')$.

We claim that $f$ is surjective, and the preimage of each element in $\Omega_1$ has
$n_p^\phi n_q^\phi$ elements, which, together with Equations \eqref{eqn:omega0} and \eqref{eqn:omega1},
implies the result.
Given $y_2\in \Omega_1$, we have $(\phi(x_1),y_2)=(y_1,y_2)\in \phi(p)$,
so there are $n_p^\phi$ elements $x_2'\in \phi^{-1}(y_2)\cap x_1p$.  For each of these elements,
we have $(\phi(x_2'),y_3)=(y_2,y_3)\in \phi(q)$, so there are $n_q^\phi$ elements
$x_3'\in \phi^{-1}(y_3)\cap x_2'q$.  Thus, for each $y_2\in \Omega_1$,
there are $n_p^\phi n_q^\phi$ pairs $(x_2',x_3')\in f^{-1}(y_2)$.
\end{proof}

\begin{cor}\label{cor:Aphi}
Suppose $S$ and $T$ are schemes on $X$ and $Y$, and $\phi\in \Hom_\Sa(S,T)$.
Then the function $\A(\phi):\A(S)\to \A(T)$ defined on basis elements by
\[\A(\phi)(\sigma_p)=n_p^\phi\sigma_{\phi(p)}\]
is an algebra homomorphism.  Moreover, if $\phi'\in \Hom_{\Sa}(S,T)$ and $\phi\tal \phi'$,
then $\A(\phi)=\A(\phi')$. 
\end{cor}

\begin{proof}
Since $\A(\phi)$ is defined only on basis elements, linearity is immediate, and we only need to check that
\[\A(\phi)(\sigma_p\cdot \sigma_q)=\A(\phi)(\sigma_p)\cdot\A(\phi)(\sigma_q)\]
holds for any $p,q\in S$, and that
\[\A(\phi)(\sigma_{1_X})=\sigma_{1_Y}.\]
Since $\phi(1_X)=1_Y$, we have
\[\A(\phi)(\sigma_{1_X})=n_{1_X}^\phi\sigma_{\phi(1_X)}=\frac{n_{1_X}}{n_{1_Y}}\sigma_{1_Y}=\sigma_{1_Y}.\]
Also, we have
\begin{eqnarray*}
\A(\phi)(\sigma_p\cdot \sigma_q)& =\A(\phi)(\sum_{r\in S} a_{pq}^r \sigma_r) & \text{definition of multiplication in }\A(S)\\
&=\sum_{r\in S} a_{pq}^r \A(\phi)(\sigma_r) & \text{linearity of }\A(\phi)\\
&=\sum_{r\in S} a_{pq}^r n_r^\phi \sigma_{\phi(r)} & \text{definition of }\A(\phi)\\
&=\sum_{t\in T} \left(\sum_{s:\phi(s)=t} a_{pq}^s n_s^\phi \right) \sigma_t &\text{regrouping terms} \\
&=\sum_{t\in T}\left( a_{\phi(p)\phi(q)}^t n_p^\phi n_q^\phi\right) \sigma_t & \text{Lemma }\ref{lem:algebra}\\
&=n_p^\phi n_q^\phi \sum_{t\in T} a_{\phi(p)\phi(q)}^t \sigma_t & \text{distributivity}\\
&=n_p^\phi  \sigma_{\phi(p)}\cdot n_q^\phi\sigma_{\phi(q)} & \text{definition of multiplication in }\A(T)\\
&=\A(\phi)(\sigma_p)\cdot \A(\phi)(\sigma_q)& \text{definition of }\A(\phi)
\end{eqnarray*}
The last claim is immediate from Definition \ref{def:alg-equiv} and Corollary \ref{cor:calculation}.
\end{proof}

\begin{rem}\label{rem:Han3}
Hanaki \cite[Theorem 3.4]{Han3} proved that if $K$ is a normal closed subset of $S$, then the natural morphism
from $S$ to $S\dm K$ induces an algebra homomorphism of adjacency algebras from $\A(S)$ to $\A(S\dm K)$ taking
$\sigma_s\in \A(S)$ to $\frac{n_s}{n_{s^K}}\sigma_{s^K}\in \A(S\dm K)$.  Thus, one could prove Corollary \ref{cor:Aphi}
as follows.  First, use Theorem \ref{thm:first-isomorphism} to factor an admissible morphism $\phi:S\to T$
as a composite $\io_{\phi(X)}\circ \bar\phi\circ\pi_K$, where $K$ is the kernel of $\phi$.
Then it is easy to check that the isomorphism $\bar\phi$ induces an algebra isomorphism taking $\sigma_{s^K}$ to
$\sigma_{\bar\phi(s^K)}$, and the inclusion $\io_{\phi(X)}$ induces a homomorphism of algebras taking
$\sigma_{\bar\phi(s^K)}$ to $\sigma_{\io_{\phi(X)}(\bar\phi(s^K))}$.  Thus, using Hanaki's result, and
since $\io_{\phi(X)}\circ\bar\phi\circ\pi_K=\phi$,
the morphism from $\A(S)$ to $\A(T)$ taking $\sigma_s$ to $\frac{n_s}{n_{s^K}}\sigma_{\phi(s)}$ is an
algebra homomorphism.  Now, since $\bar\phi$ and $\io_{\phi(X)}$
both preserve valencies, we have $n_{s^K}=n_{\phi(s)}$, so $\frac{n_s}{n_{s^K}}=n_s^{\phi}$ by
Corollary \ref{cor:calculation}.  Thus, this algebra homomorphism coincides with the one defined in Corollary
\ref{cor:Aphi} above.
\end{rem}

\begin{cor}\label{cor:Algfunctor} Suppose $S$, $T$, and $U$ are schemes on $X$, $Y$, and $Z$, and
$\phi\in \Hom_{\Sa}(S,T), \psi\in \Hom_{\Sa}(T,U)$.  Then, referring to the notation introduced in
Corollary \ref{cor:Aphi}, we have
\[\A(\psi\circ \phi)=\A(\psi)\circ\A(\phi):\A(S)\to \A(U).\]  Also,
if $\id_S\in \Hom_{\Sa}(S,S)$ is the identity, then $\A(\id_S):\A(S)\to \A(S)$ is the identity.
Thus, $\A$ is a functor from $\Sa$ to $\Alg$.
\end{cor}

\begin{proof} Suppose $s\in S$, $x_0\in X$, $z\in Z$, and $(\psi(\phi(x_0)),z)\in\psi(\phi(s))$.  
Then there are $n_{\phi(s)}^\psi$ elements $y\in Y$ such that $\psi(y)=z$
and $(\phi(x_0),y)\in \phi(s)$.  For each such element $y$, there are $n_s^\phi$ elements $x\in X$ such
that $\phi(x)=y$ and $(x_0,x)\in s$.  Thus, 
\[|\{(x,y)\in X\times Y:\phi(x)=y, \psi(y)=z, (x_0,x)\in s, (\phi(x_0),y)\in \phi(s)\}|=n_s^\phi\cdot n_{\phi(s)}^\psi.\]
On the other hand,
\[|\{x\in X:\psi\phi(x)=z, (x_0,x)\in s\}|=n_s^{\psi\phi}.\]
The function taking $x$ to $(x,\phi(x))$ induces a bijection
between the two sets displayed on the left side of these equations,
so $n_s^{\psi\phi}=n_s^\phi\cdot n_{\phi(s)}^\psi$.

Now, we have
\begin{equation*}
\begin{split}
\A(\psi\phi)(\sigma_s)&=n_s^{\psi\phi}\sigma_{\psi\phi(s)}\\
&=n_s^\phi n_{\phi(s)}^\psi \sigma_{\psi\phi(s)}\\
&=n_s^\phi\A(\psi)(\sigma_{\phi(s)})\\
&=\A(\psi)(n_s^\phi \sigma_{\phi(s)})\\
&=\A(\psi)(\A(\phi)(\sigma_s))
\end{split}
\end{equation*}
The second claim follows since $n_s^{\id}=\frac{n_s}{n_s}=1$ for any $s\in S$, and the final statement follows from the first two.
\end{proof}

\begin{rem}\label{rem:Afactor} The functor $\A$ factors through $F:\Sa\to \Sal$ (see Definition
\ref{def:Salg}) by the last claim of Corollary
\ref{cor:Aphi}.  We will denote this factorization by $\Aa:\Sal\to \Alg$.  That is, $\Aa\circ F=\A$.
\end{rem}

\begin{lem}\label{lem:adjacency}
Let $\mathbb C[-]$ denote the functor from the category of finite groups to the category of algebras over
$\mathbb C$.  For each group $G$, let $\eta_G:\mathbb C[G]\to \A(S(G))$ denote the isomorphism
of algebras described in Remark \ref{rem:adjacency}.
Then $\eta$ is a natural isomorphism.
\end{lem}

\begin{proof}
We must show that $\A(S(\phi))\circ\eta_G=\eta_K\circ \mathbb C[\phi]$ when $\phi:G\to K$ is a homomorphism
of groups.  Since the elements $g\in G$ form a basis for $\mathbb C[G]$, it suffices to show the equation holds when
both sides are applied to an arbitrary $g\in G$.
Then we have
\begin{eqnarray*}
\A(S(\phi))(\eta_G(g))&=\A(S(\phi))(\sigma_{g\tilde{\:}})&\text{Remark }\ref{rem:adjacency}\\
&=n_{g\tilde{\:}}^{S(\phi)}\sigma_{S(\phi)(g\tilde{\:})}&\text{Corollary }\ref{cor:Aphi}\\
&=n_{g\tilde{\:}}^{S(\phi)}\sigma_{\phi(g)\tilde{\:}}&\text{Corollary }\ref{cor:Sfunctor}\\
&=\frac{n_{g\tilde{\:}}}{n_{\phi(g)\tilde{\:}}}\sigma_{\phi(g)\tilde{\:}}&\text{Corollary }\ref{cor:calculation}\\
&=\sigma_{\phi(g)\tilde{\:}}&\text{since }S(G)\text{ and }S(K)\text{ are thin schemes}\\
&=\eta_K(\phi(g))&\text{Remark }\ref{rem:adjacency}\\
&=\eta_K(\mathbb C[\phi](g))&\text{Definition of }\mathbb C[\phi].
\end{eqnarray*}
\end{proof}

\section{Products of association schemes}\label{sec:product}

If $T$ and $U$ are schemes on $Y$ and $Z$, define
\[\zeta_{T,U}:T\times U\to {\cal P}((Y\times Z)^{\times 2})\] by \[\zeta_{T,U}(t,u)=
\{((y_1,z_1),(y_2,z_2))\in (Y\times Z)^{\times 2} :(y_1,y_2)\in t\text{ and }
(z_1,z_2)\in u\}.\]
If $t\in T$ and $u\in U$, we will write $[t,u]$ for $\zeta_{T,U}(t,u)$.
The set of subsets $[t,u]$ with $t\in T, u\in U$ then forms a scheme on 
$Y\times Z$, cf. \cite[Theorem 7.2.3(i)]{PHZ}.  Of course $[t,u]^*=[t^*,u^*]$
and $1_{Y\times Z}=[1_Y,1_Z]$.
We denote this scheme $T\boxtimes U$, and call it the direct product of the schemes $T$ and $U$.

\begin{rem}\label{rem:product-structure}
Suppose $T_1$ and $T_2$ are schemes on $Y_1$ and $Y_2$, and suppose
$p_1, q_1, r_1\in T_1$ and $p_2, q_2, r_2\in T_2$.  Then 
from \cite[Theorem 7.2.3(ii)]{PHZ}), we have
\[a_{[p_1,p_2]\ [q_1, q_2]}^{[r_1, r_2]}
=a_{p_1 q_1}^{r_1}\cdot a_{p_2 q_2}^{r_2}.\]

It follows that if $t_1\in T_1$ and $t_2\in T_2$, then
\[n_{[t_1,t_2]}
=a_{[t_1,t_2]\ [t_1^*,t_2^*]}^{[1,1]}
=a_{t_1 t_1^*}^{1}\cdot a_{t_2 t_2^*}^{1}=n_{t_1}\cdot n_{t_2}.\]
\end{rem}

\begin{rem}\label{rem:coincide}
If $G$ and $H$ are groups, then the schemes $S(G\times H)$ and $S(G)\boxtimes S(H)$
on $G\times H$ are identical.  Indeed, if $g\in G$ and $h\in H$, then
\[{(g,h)}\tilde{\:}=\{(g_1,h_1),(g_2,h_2)\in (G\times H)^{\times 2}:g_2=g_1g, h_2=h_1h\}
=[g\tilde{\:}, h\tilde{\:}]\]
so both schemes consist of the same subsets of $(G\times H)^{\times 2}$.
\end{rem}

Now suppose $S_1$ and $S_2$ are schemes on $X_1$ and $X_2$, while $T_1$ and $T_2$
are schemes on $Y_1$ and $Y_2$.  Given morphisms $\phi_1$ from $S_1$ to $T_1$ and $\phi_2$
from $S_2$ to $T_2$, we get a product morphism $\phi_1\boxtimes\phi_2$ from $S_1\boxtimes S_2$
to $T_1\boxtimes T_2$.  This is defined as follows:
\[(\phi_1\boxtimes\phi_2)(x_1,x_2)=(\phi_1(x_1),\phi_2(x_2)), \text{ for }x_1\in X_1, x_2\in X_2,\]
\[(\phi_1\boxtimes\phi_2)([s_1,s_2])
=[\phi_1(s_1),\phi_2(s_2)],\text{ for }s_1\in S_1, s_2\in S_2.\]

\begin{lem}\label{lem:product} Suppose $S_1, S_2, T_1$ and $T_2$ are schemes on finite sets
$X_1, X_2, Y_1$ and $Y_2$.  If $\phi_i\in \Hom_{\Sa}(S_i,T_i)$ for $i\in \{1,2\}$,
then $\phi_1\boxtimes\phi_2$ is an admissible morphism
from $S_1\boxtimes S_2$ to $T_1\boxtimes T_2$.
\end{lem}

\begin{proof}
Suppose \[((\phi_1\boxtimes\phi_2)(x_1,x_2),(y_1,y_2))\in 
(\phi_1\boxtimes \phi_2)([s_1,s_2]).\]
By definition, this means
$(\phi_1(x_1),y_1)\in \phi_1(s_1)$ and $(\phi_2(x_2),y_2)\in \phi_2(s_2)$.  Since $\phi_1$ and $\phi_2$
are admissible, we can find $x'_1\in X_1$ and $x'_2\in X_2$ such that $\phi_1(x'_1)=y_1$, $\phi_2(x'_2)=y_2$,
$(x_1,x_1')\in s_1$ and $(x_2,x_2')\in s_2$.  But then, we have
\[(\phi_1\boxtimes\phi_2)(x'_1,x'_2)=(y_1,y_2)\text{ and }
((x_1,x_2),(x'_1,x'_2))\in [s_1,s_2].\]
\end{proof}

\begin{rem}\label{rem:prodfunctor} Suppose $S_i$, $T_i$, and $U_i$ are schemes on
finite sets $X_i$, $Y_i$, and $Z_i$
for $i\in \{1,2\}.$
If $\phi_i\in \Hom_{\Sa}(S_i,T_i)$ and $\psi_i\in \Hom_{\Sa}(T_i,U_i)$ for $i\in \{1,2\}$,
then it is easy to see that
\[(\psi_1\boxtimes\psi_2)\circ(\phi_1\boxtimes\phi_2)=(\psi_1\circ\phi_1)\boxtimes (\psi_2\circ\phi_2).\]
Also, if $\id_{S_i}$ is the identity on $S_i$ for $i\in \{1,2\}$, then it is easy to see that
$\id_{S_1}\boxtimes\id_{S_2}=\id_{S_1\boxtimes S_2}$.  Thus, we have a product
functor $\Sa\times \Sa\to \Sa$
taking the pair of schemes $(S_1, S_2)$ to $S_1\boxtimes S_2$ and the pair of morphisms
$(\phi_1,\phi_2)$ to $\phi_1\boxtimes\phi_2$.
\end{rem}

We do {\it not} claim that $S_1\boxtimes S_2$ is a product object in the category $\Sa$.
That is, if $T$ is a scheme on a finite set $Y$, and we are given morphisms from $\phi_i\in \Hom_{\Sa}(T,S_i)$,
we do not claim that there is a morphism in $\Hom_{\Sa}(T,S_1\boxtimes S_2)$ satisfying the
usual universal property.  Indeed, this is generally not the case, as we will see in Lemma
\ref{lem:Delta} below.

\begin{defn}\label{def:Delta}
Suppose $S$ is a scheme on a finite set $X$.  Let $\Delta_S$ be the morphism from $S$ to $S\boxtimes S$
defined as follows:
\[\Delta_S(x)=(x,x)\text{ for }x\in X, \Delta_S(s)=[s,s]\text{ for }s\in S.\]
\end{defn}

Indeed, if $(x,x')\in s$, then $((x,x),(x',x'))\in \Delta_S(s)$, so $\Delta_S$ is a morphism of schemes.

\begin{lem}\label{lem:Delta} Given a scheme $S$ on a finite set $X$, the morphism $\Delta_S$ from $S$ to $S\boxtimes S$
is admissible if and only if $S$ is thin.
\end{lem}

\begin{proof} Suppose given $s\in S$ and $x_0\in X$, and suppose $x_1,x_2\in x_0s$.
Then \[(\Delta_S(x_0),(x_1,x_2))=((x_0,x_0),(x_1,x_2))\in [s,s]=\Delta_S(s).\]
If $\Delta_S$ is admissible, then there exists $x\in X$
such that $\Delta_S(x)=(x_1,x_2)$, whence $x_1=x=x_2$.  Thus, if $\Delta_S$ is admissible, then $n_s=1$ for all $s\in S$.
Conversely, if $n_s=1$ for all $s\in S$, then $(\Delta_S(x_0),(x_1,x_2))\in \Delta_S(s)$ implies
$(x_0,x_1)\in s$ and $(x_0,x_2)\in s$,
so $x_1=x_2$.  This implies $\Delta_S(x_1)=(x_1,x_2)$ and $(x_0,x_1)\in s$, so $\Delta_S$ is admissible.
\end{proof}

\begin{lem}\label{lem:Delta'} Suppose $S$ is a scheme on a finite set $X$ and $T$ is a thin scheme on
a finite set $Y$, and suppose
$\phi$ is a morphism from $S$ to $T$.  Then we have
\[(\phi\times\id_S)\circ\Delta_S\in \Hom_{\Sa}(S,T\times S).\]
\end{lem}

\begin{proof}
Let $\delta_S^\phi=(\phi\times\id_S)\circ \Delta_S$.  Thus, $\delta_S^\phi(x)=(\phi(x),x)$ and $\delta_S^\phi(s)=[\phi(s),s]$.
Suppose $(\delta_S^\phi(x_0),(y,x))\in \delta_S^\phi(s)$.  Then $(\phi(x_0),y)\in \phi(s)$ and $(x_0,x)\in s$. 
Since $\phi$ is a morphism of schemes, $(x_0,x)\in s$ implies
$(\phi(x_0),\phi(x))\in \phi(s)$.  Thus, since $\phi(s)\in T$ is thin and $(\phi(x_0),y)\in \phi(s)$,
we must have $y=\phi(x)$.  Therefore, $\delta_S^\phi(x)=(\phi(x),x)=(y,x)$
and $(x_0,x)\in s$, so $\delta_S^\phi$ is admissible.
\end{proof}

\begin{rem}\label{rem:Delta}
If $S$ is thin, then by Lemma \ref{lem:Delta'}, $\Delta_S$ is admissible, since we may choose $\phi=\id_S$.
If $G$ is a finite group, then let $\Delta_G:G\to G\times G$ be given by $\Delta_G(g)=(g,g)$.
Thus, we have two admissible morphisms
$\Delta_{S(G)}$ and $S(\Delta_G)$ from $S(G)$ to $S(G\times G)=S(G)\boxtimes S(G)$
(see Remark \ref{rem:coincide}).  If $g\in G$, then by Corollary \ref{cor:Sfunctor},
$S(\Delta_G)(g)=(g,g)$ and $S(\Delta_G)(g\tilde{\:})={(g,g)}\tilde{\:}$.  Likewise, by Definition
\ref{def:Delta}, $\Delta_{S(G)}(g)=(g,g)$ and $\Delta_{S(G)}(g\tilde{\:})=[g\tilde{\:}, g\tilde{\:}]$.
By Remark \ref{rem:coincide}, $[g\tilde{\:}, g\tilde{\:}]={(g,g)}\tilde{\:}$,
so $S(\Delta_G)=\Delta_{S(G)}$. 
\end{rem}

We now show that $\A$ takes products of schemes to tensor products of algebras.

\begin{prop}\label{prop:prod}
Suppose $S$ and $T$ are schemes on $X$ and $Y$.  Then there is a unique algebra isomorphism
$\mu_{S,T}:\A(S)\otimes \A(T)\to \A(S\boxtimes T)$ satisfying
\begin{equation}\label{eqn:mu}
\mu_{S,T}(\sigma_s\otimes \sigma_t)=\sigma_{[s,t]}
\end{equation}
Moreover, the homomorphisms $\mu_{S,T}$ are natural; more precisely,
if $\phi\in \Hom_{\Sa}(S,S')$ and $\psi\in \Hom_{\Sa}(T,T')$, then
\[\A(\phi\boxtimes\psi)\circ\mu_{S,T}=\mu_{S',T'}\circ (\A(\phi)\otimes \A(\psi)).\]
\end{prop}

\begin{proof}
The set $\{\sigma_s\otimes \sigma_t\}_{s\in S, t\in T}$ is a basis for
$\A(S)\otimes \A(T)$ as a complex vector space.  Thus, there is a unique
linear transformation $\mu_{S,T}$ satisfying equation \eqref{eqn:mu}.  We are left to show
that $\mu_{S,T}$ is an isomorphism of vector spaces and a homomorphism of algebras.

To see that $\mu_{S,T}$ is an isomorphism of vector spaces, first observe that
$\mu_{S,T}$ is surjective, since each scheme element in $S\boxtimes T$ is of the form
$[s,t]$ for some $s\in S$ and $t\in T$.  Since the dimension
of the domain and codomain of $\mu_{S,T}$ are both equal to $|S|\cdot |T|$,
$\mu_{S,T}$ is an isomorphism.

To see that $\mu_{S,T}$ is a homomorphism of algebras, first, note that
\[\mu_{S,T}(\sigma_{1_X}\otimes \sigma_{1_Y})=\sigma_{[1_X,1_Y]}=\sigma_{1_{X\times Y}},\]
so $\mu_{S,T}$ preserves the multiplicative identity.
Now, suppose $s_1, s_2\in S$ and $t_1,t_2\in T$.  Then
\begin{equation*}
\begin{split}
\mu_{S,T}\left((\sigma_{s_1}\otimes \sigma_{t_1})\cdot (\sigma_{s_2}\otimes \sigma_{t_2})\right)
&=\mu_{S,T}\left(\left(\sigma_{s_1}\cdot \sigma_{s_2}\right)\otimes
 \left(\sigma_{t_1}\cdot \sigma_{t_2}\right)\right)\\
&=\mu_{S,T}\left(\left(\sum_{s\in S} a_{s_1s_2}^s \sigma_s\right)\otimes
\left(\sum_{t\in T} a_{t_1t_2}^t\sigma_t\right)\right)\\
&=\sum_{s\in S, t\in T} a_{s_1s_2}^s a_{t_1t_2}^t \mu_{S,T}(\sigma_s\otimes\sigma_t)\\
&=\sum_{[s,t]\in S\boxtimes T} a_{[s_1, t_1]\ [s_2,t_2]}^{[s,t]}
\sigma_{[s,t]}\quad  (\text{cf. Remark } \ref{rem:product-structure})\\
&=\sigma_{[s_1,t_1]}\cdot \sigma_{[s_2,t_2]}\\
&=\mu_{S,T}(\sigma_{s_1}\otimes \sigma_{t_1})\cdot \mu_{S,T}(\sigma_{s_2}\otimes \sigma_{t_2}).
\end{split}
\end{equation*}

Lastly, we prove naturality.  Note that $\phi\boxtimes\psi$ is admissible
by Lemma \ref{lem:product}.  Now, if $s\in S$ and $t\in T$, then
we have
\begin{equation*}
\begin{split}
(\mu_{S',T'}\circ (\A(\phi)\otimes \A(\psi)))(\sigma_s\otimes\sigma_t)&=
\mu_{S',T'}(n_s^\phi \sigma_{\phi(s)}\otimes n_t^\psi \sigma_{\psi(t)})\quad
(\text{Corollary }\ref{cor:Aphi})\\
&=(n_s^\phi\cdot n_t^\psi)\sigma_{[\phi(s),\psi(t)]}\\
&=\frac{n_s\cdot n_t}{n_{\phi(s)}\cdot n_{\psi(t)}}\sigma_{[\phi(s),\psi(t)]}
\quad
(\text{Corollary }\ref{cor:calculation})\\
&=\frac{n_{[s,t]}}{n_{\phi\boxtimes\psi([s,t])}}\sigma_{[\phi(s),\psi(t)]}
\quad
(\text{Remark }\ref{rem:product-structure})\\
&=n_{[s,t]}^{\phi\boxtimes\psi}\sigma_{(\phi\boxtimes\psi)([s,t])}
\quad 
(\text{Corollary }\ref{cor:calculation})\\
&=\A(\phi\boxtimes\psi)(\sigma_{[s,t]})\\
&=(\A(\phi\boxtimes\psi)\circ\mu_{S,T})(\sigma_s\otimes\sigma_t).
\end{split}
\end{equation*}
\end{proof}

\section{Hopf structures}\label{sec:Hopf}

We recall that a (complex) bialgebra is a unital algebra $A$ over $\mathbb C$,
equipped with a comultiplication $\Delta:A\to A\otimes A$ and a counit $\epsilon:A\to \mathbb C$,
both of which are algebra homomorphisms, satisfying
coassociativity and counitality conditions.  Here, coassociativity means that
\[(\Delta\otimes \id_A)\circ \Delta=(\id_A\otimes \Delta)\circ\Delta\]
and counitality means that $(\epsilon\otimes \id_A)\circ\Delta$ takes $a\in A$ to
$1\otimes a\in \mathbb C\otimes A$.  We say that $A$ is a Hopf algebra if it also comes
equipped with a (necessarily unique) antipode map $S:A\to A$ such that
\[\nabla\circ (1\otimes S)\circ \Delta=\nabla\circ (S\otimes 1)\circ \Delta=\eta\circ\epsilon\]
where $\nabla:A\otimes A\to A$ is multiplication and $\eta:\mathbb C\to A$ is the unit.
We say $A$ is cocommutative if $\tau\circ\Delta=\Delta$, where $\tau:A\otimes A\to A\otimes A$
takes $a\otimes b$ to $b\otimes a$.
See \cite{ABE} for more details on Hopf algebras.  The following is
Example 2.7 in \cite{ABE}.

\begin{example}\label{ex:group-Hopf}
For a finite group $G$, $\mathbb C[G]$ is a Hopf algebra with comultiplication $\Delta_G$
taking $g$ to $g\otimes g$, counit $\epz_G$ taking $g$ to $1$, and antipode $S_G$ taking $g$ to $g^{-1}$.
The comultiplication is clearly cocommutative.
\end{example}

\begin{prop}\label{prop:Hopf} If $T$ is a thin scheme, then the algebra $\A(T)$ has the structure of a
cocommutative Hopf algebra.  The counit $\epz_T$ takes $\sigma_t$ to $1$; the antipode
$S_T$ takes $\sigma_t$ to $\sigma_{t^*}$; and the
comultiplication $\Delta_T$ takes $\sigma_t$ to $\sigma_t\otimes\sigma_t$.  This Hopf algebra
is isomorphic to the Hopf algebra $\mathbb C[\Res(T)]$ of Example \ref{ex:group-Hopf}.
\end{prop}

\begin{proof}
Let $\phz$ be a representative in $\Hom_{\Sa}(T,S(\Res(T)))$ of the unit of the $(\Res,S_a)$-adjunction
in $\Hom_{\Sal}(T,S_a(\Res(T)))$, described in Remark \ref{rem:unitSQ}.  Thus,
$\phz(t)=\{t^{T'}\}\tilde{\:}=\{t\}\tilde{\:}$, since $T'$ is trivial.  Then
$\phz$ is an isomorphism, so we have an isomorphism of algebras
\[\xymatrix{
{\Phi:\A(T)}\ar[r]^-{\A(\phz)} &
{\A(S(\Res(T))}\ar[r]^-{\eta_{\Res(T)}^{-1}} &
{\mathbb C[\Res(T)]}}\]
where $\eta_{\Res(T)}$ is the isomorphism defined in Lemma \ref{lem:adjacency}.
For $t\in T$, we have
\[\eta_{\Res(T)}(\Phi(\sigma_t))=\A(\phz)(\sigma_t)=\frac{n_t}{n_{\phz(t)}}\sigma_{\phz(t)}
=\sigma_{\phz(t)}=\sigma_{\{t\}\tilde{\:}}.\]
Since $\eta_{\Res(T)}(\{t\})=\sigma_{\{t\}\tilde{\:}}$,
$\Phi(\sigma_t)$ is the generator $\{t\}\in \mathbb C[\Res(T)]$.
From here, we immediately obtain
\[\Delta_{\Res(T)}\circ \Phi= (\Phi\otimes\Phi)\circ\Delta_{T},
\quad \epz_{T}=\epz_{\Res(T)}\circ \Phi,\quad\text{and}\quad
\Phi\circ S_{T}=S_{\Res(T)}\circ \Phi.\]
Since $\Phi$ is an isomorphism of algebras, it follows that $\A(T)$ is a cocommutative Hopf algebra
and $\Phi$ is an isomorphism of Hopf algebras.
\end{proof}

\begin{rem}\label{rem:repring}
For any algebra $A$, the set of isomorphism classes of finite dimensional
$A$-modules forms a semigroup,
with sum given by direct sum of modules.  If $V$ and $W$ are $A$-modules, then $V\otimes W$
is an $A\otimes A$-module, so one does not typically get a semiring structure on these isomorphism
classes.  However, when $A$ is a Hopf algebra (e.g. $A=\mathbb C[G]$ for a group $G$),
then one may pull back the action map (of $A\otimes A$ on $V\otimes W$)
along the comultiplication map to make $V\otimes W$ into an $A$-module.  Coassociativity
of the comultiplication implies that if $U$, $V$, and $W$ are $A$-modules,
then $U\otimes (V\otimes W)$ and $(U\otimes V)\otimes W$ are isomorphic $A$-modules.
If $A$ is cocommutative,
then the canonical isomorphism $V\otimes W\iso W\otimes V$ is an isomorphism of $A$-modules.
Moreover, the counit of $A$ determines a one-dimensional module, and counitality implies that the isomorphism
class of this module is a multiplicative identity.
Thus, the set of isomorphism classes of modules over a Hopf algebra $A$ forms a commutative, unital
semiring.  By taking formal differences
of elements in this semiring, one obtains a commutative unital ring.  For example, for a finite group $G$,
the formal differences
of isomorphism classes of $\mathbb C[G]$-modules form a commutative unital ring $R(G)$,
the representation ring of $G$.  Using Proposition \ref{prop:Hopf}, we can apply this procedure to a
thin scheme $T$; then, the representation ring $R(T)$ of $T$ is isomorphic to $R(\Res(T))$.
\end{rem}

Now suppose $A$ is a Hopf algebra, with comultiplication map $\Delta_A:A\to A\otimes A$,
counit $\epz_A:A\to \mathbb C$, multiplication $\nabla_A:A\otimes A\to A$, and unit
$\eta_A:\mathbb C\to A$.  A left $A$-comodule is a vector space $M$ equipped with a linear
coaction map $\rho_M:M\to A\otimes M$ which satisfies coassociativity and counitality.  More concretely, we require that
\begin{equation}\label{eq:coassociative}
(\id_{A}\otimes \rho_M)\circ \rho_M=(\Delta_A\otimes \id_M)\circ \rho_M,
\end{equation}
and for any $m\in M$,
\begin{equation}\label{eq:counital}
((\epz_A\otimes \id_M)\circ\rho_M)(m)=1\otimes m\in \mathbb C\otimes M.
\end{equation}

A left $A$-comodule algebra is an algebra, which is also a left $A$-comodule, such that the
coaction $\rho_M$ is an algebra homomorphism.  More concretely, letting $\nabla_M$
and $\eta_M$ denote the multiplication and unit of $M$, we require that
\begin{equation}\label{eq:algebra}
\rho_M\circ\nabla_M=(\nabla_A\otimes \nabla_M)\circ (\id_A\otimes \tau_{M,A}\otimes \id_M)\circ
(\rho_M\otimes \rho_M),
\end{equation}
and
\begin{equation}\label{eq:unital}
\rho_M\circ \eta_M=(\eta_A\otimes \eta_M)\circ \Delta_{\Cx},
\end{equation}
where $\tau_{M,A}:M\otimes A\to A\otimes M$ and
$\Delta_{\Cx}:\mathbb C\to \mathbb C\otimes \mathbb C$
are the canonical isomorphisms.
See \cite[Section 3.2]{ABE} for an equivalent formulation of this definition.

Now fix a morphism $\phi$ from a scheme $S$ to a thin scheme $T$ (so $\phi$ is admissible
by Lemma \ref{lem:to-thin}).  Then let $\rho_S:\A(S)\to \A(T)\otimes \A(S)$ and
$\epz_S:\A(S)\to \A(T)$ be defined on generators
$\sigma_p\in \A(S)$ by \[\rho_S(\sigma_p)=\sigma_{\phi(p)}\otimes \sigma_p,\quad
\epz_S(\sigma_p)=n_p\sigma_{\phi(p)}\]

\begin{prop}\label{prop:Hopfcom}
The function $\epz_S$ is an algebra homomorphism, and
the function $\rho_S$ endows $\A(S)$ with the structure of a left $\A(T)$-comodule algebra.
\end{prop}

\begin{proof}
First, by Corollary \ref{cor:Aphi}, $\A(\phi)$ is a homomorphism of algebras, and
$\A(\phi)(\sigma_p)=n_p^{\phi}\sigma_{\phi(p)}$.
By Corollary \ref{cor:calculation}, and since $\phi(p)$ is thin, we have
$n_p^\phi=\frac{n_p}{n_{\phi(p)}}=n_p$.  Therefore, $\epz_S=\A(\phi)$, which
is an algebra homomorphism.

Now, by Proposition \ref{prop:prod}, we have
\[\mu_{T,S}(\rho_S(\sigma_p))=\mu_{T,S}(\sigma_{\phi(p)}\otimes \sigma_p)=\sigma_{[\phi(p),p]}.\]
On the other hand, the function $\delta_S^\phi:=(\phi\times \id_S)\circ \Delta_S$ is admissible
by Lemma \ref{lem:Delta'}, so $\A(\delta_S^\phi)$ is an algebra homomorphism.  We have
\[\A(\delta_S^\phi)(\sigma_p)=n_p^{\delta_S^\phi}\sigma_{\delta_S^\phi(p)}
=\frac{n_p}{n_{[\phi(p),p]}}\sigma_{[\phi(p),p]}.\]
By Remark \ref{rem:product-structure}, $n_{[\phi(p),p]}=n_{\phi(p)}n_p=n_p$, since $\phi(p)$ is thin.
Thus, \[\A(\delta_S^\phi)(\sigma_p)=\sigma_{[\phi(p),p]}=\mu_{T,S}(\rho_S(\sigma_p)).\]
Now, $\mu_{T,S}$ is an isomorphism of algebras by Proposition \ref{prop:prod}, so
$\mu_{T,S}^{-1}\circ \A(\delta_S^\phi)=\rho_S$ is an algebra homomorphism.

Equations (\ref{eq:coassociative}) and (\ref{eq:counital}) are easy computations.  For $\sigma_p\in \A(S)$, we have
\begin{equation*}
\begin{split}
((\id_{\A(T)}\otimes \rho_S)\circ \rho_S)(\sigma_p)&=
(\id_{\A(T)}\otimes \rho_S)(\sigma_{\phi(p)}\otimes \sigma_p)\\
&=\sigma_{\phi(p)}\otimes \sigma_{\phi(p)}\otimes \sigma_p\\
&=(\Delta_T\otimes \id_{\A(S)})(\sigma_{\phi(p)}\otimes \sigma_p)\\
&=(\Delta_T\otimes \id_{\A(S)})(\rho_S(\sigma_p)).
\end{split}
\end{equation*}
Likewise, we have
\[((\epz_T\otimes \id_{\A(S)})\circ\rho_S)(\sigma_p)=
(\epz_T\otimes \id_{\A(S)})(\sigma_{\phi(p)}\otimes \sigma_p)
=1\otimes \sigma_p.\]
\end{proof}

\begin{rem}\label{rem:repmod}
If $A$ is a (complex) cocommutative Hopf algebra and $M$ is a left $A$-comodule algebra,
then the set of isomorphism classes
of finite dimensional $M$-modules forms a semigroup as in Remark \ref{rem:repring}.  Now,
if $V$ is an $A$-module and $W$ is an $M$-module, then $V\otimes W$ is an $A\otimes M$-module;
pulling back along the coaction $\rho_M:M\to A\otimes M$ makes $V\otimes W$ into an
$M$ module.  Thus, the semigroup of finite dimensional $M$-modules is a semimodule
over the semiring of finite dimensional $A$-modules; taking formal differences as before,
we obtain a module over the commutative ring described in Remark \ref{rem:repring}.
Thus, if $\phi$ is a morphism from a scheme $S$ to a thin scheme $T$, then
by Proposition \ref{prop:Hopfcom}, the formal differences of isomorphism classes
of finite dimensional $\A(S)$-modules, which we denote $R(S)$,
form a module over the representation ring $R(T)$ of $T$.  In particular, we may
take $T$ to be $S\dm O^{\vartheta}(S)$, and take $\phi$ to be the natural
quotient morphism.  Since $O^{\vartheta}(T)$ is trivial,
$T\dm O^{\vartheta}(T)$ is canonically isomorphic to $T$, so, by Definition \ref{def:residue},
$\Res(T)$ and $\Res(S)$ are canonically isomorphic.
Thus, we have defined an $R(\Res(S))$-module structure on $R(S)$.
\end{rem}

Hanaki \cite{Han2} proved that the product of characters of two representations
of a finite scheme $S$ is itself a character if one of the two characters contains the thin
residue $O^{\vartheta}(S)$ in its kernel, which is equivalent to the above remark.
As above, the key step is to prove that the map $\rho_S$ is a morphism of algebras, which
Hanaki does by direct computation (in \cite[Proposition 3.1]{Han2}),
in the primary case when $\phi$ is the quotient map from $S$ to $S\dm O^{\vartheta}(S)$.
Our proof shows the central role of admissible morphisms in this result.  That is, the essential
reason that $\rho_S$ is an algebra morphism is that the map $(\phi\times \id_S)\circ \Delta_S$
is an admissible morphism of schemes.  Likewise, the failure of the morphism
$\A(S)\to \A(S)\otimes \A(S)$ taking $\sigma_s$ to $\sigma_s\otimes \sigma_s$
to be an algebra homomorphism can be attributed to the lack of admissibility of the morphism $\Delta_S$.

\end{document}